\documentclass[a4paper,reqno,oneside,11pt]{amsart}

\usepackage[english]{babel}
\usepackage{amsmath}
\usepackage{amssymb}

\usepackage{mathrsfs}
\usepackage{amssymb}
\usepackage{amsthm}
\usepackage{enumerate}
\usepackage{ifthen}
\usepackage{bbm}
\usepackage{color}
\usepackage{xcolor}
\usepackage{graphicx}
\usepackage{hyperref}
\usepackage{soul}
\usepackage{geometry}
\newcommand{\margnote}[1]{
\ifthenelse{\boolean{shownotes}}%
{\marginpar{\raggedright\tiny\texttt{#1}}}%
{}%
}

\newcommand{\hole}[1]{
\ifthenelse{\boolean{shownotes}}%
{\begin{center} \fbox{ \rule {.25cm}{0cm}
\rule[-.1cm]{0cm}{.4cm} \parbox{.85\textwidth}{\begin{center}
\texttt{#1}\end{center}} \rule {.25cm}{0cm}}\end{center}}
{}
}
\newtheorem{thm}{Theorem}[section]
\theoremstyle{plain}
\newtheorem{lemma}[thm]{Lemma}
\theoremstyle{plain}
\newtheorem*{thm*}{Theorem}
\theoremstyle{plain}
\newtheorem{prop}[thm]{Proposition}
\theoremstyle{plain}

\theoremstyle{definition}
\newtheorem{definition}[thm]{Definition}    
\theoremstyle{remark}
\newtheorem{remark}{Remark}

\newcommand{\T}{\mathbb{T}}
\newcommand{\R}{\mathbb{R}}

\newcommand{\e}{\varepsilon}
\newcommand{\re}{\rho_{\varepsilon}}

\newcommand{\me}{\mu_{\varepsilon}}
\newcommand{\phie}{\phi_{\varepsilon}}

\newcommand{\Fe}{F_{\varepsilon}}

\newcommand{\weakto}{\rightharpoonup}

\newcommand{\del}{\partial}
\newcommand{\dx}{\partial_{x}}
\newcommand{\ddx}{\partial^2_{x}}
\newcommand{\dxk}{\partial^k_{x}}
\newcommand{\dxkpo}{\partial^{k+1}_{x}}
\newcommand{\lese}{\lesssim_{\e}}
\newcommand{\les}{\lesssim}
\newcommand{\ue}{u_{\e}}
\newcommand{\deps}{\delta_{\e}}
\newcommand{\A}{\mathcal{Q}}
\newcommand{\tk}{\tilde{\kappa}_{\e}(\rho)}

\newcommand{\bm}{\bar{\mu}(\rho)}

\numberwithin{equation}{section}

\begin{document}

\title[\empty]{On the existence of non-negative weak solutions for $1D$ fourth order equations of gradient flow type}

\author[S. Georgiadis]{Stefanos Georgiadis}
\address[S. Georgiadis]{Division of Science and Mubadala Arabian Center for Climate and Environmental Science\\ New York University
Abu Dhabi \\ United Arab Emirates}
\email[]{\href{stefanos.georgiadis@nyu.edu}{stefanos.georgiadis@nyu.edu}}

\author[S. Spirito]{Stefano Spirito}
\address[S. Spirito]{DISIM - Dipartimento di Ingegneria e Scienze dell'Informazione e Matematica\\ Universit\`a  degli Studi dell'Aquila \\Via Vetoio \\ 67100 L'Aquila \\ Italy}
\email[]{\href{stefano.spirito@univaq.it}{stefano.spirito@univaq.it}}

\begin{abstract}
  In this paper, we consider a family of one-dimensional fourth order evolution equations arising as gradient flows of the Korteweg energy, i.e. the $L^2$-norm of the first derivative of some power of the density. This family of equations generalizes the Quantum-Drift-Diffusion equation and the Thin-Film equation. We prove the global-in-time existence of {\em non-negative} weak solutions without requiring any upper bound on the exponent of the power of the density in the energy. 
 \end{abstract}

\maketitle
\section{Introduction}
\subsection{Main result and state of the art} Let $T>0$ and $\T$ be the $1$-dimensional flat-torus. For $\alpha\in\R\setminus\{0\}$, we consider in $(0,T)\times\T$ the fourth-order equation for the positive unknown $\rho:(0,T)\times\T\to\R^+$ given by
\begin{equation} \label{eq:B}
\partial_t \rho +\frac{1}{\alpha} \, \partial_x \left( \rho \, \partial_x \left( \rho^{\alpha - 1} \, \partial_x^2 \rho^\alpha \right) \right)=0.
\end{equation}
Equation \eqref{eq:B} has been proposed by Denzler and McCann in \cite{DM} in the study of certain diffusion processes, and then studied by Matthes, McCann, and Savar\'e in \cite{MMS}. We remark that \eqref{eq:B} can be equivalently written for $\beta=2\alpha-2$ as 
\begin{equation}\label{eq:gfb}
\partial_t\rho+\,\del_x\left(\rho\del_x\left(\del_x(\rho^{\beta}\del_x\rho)-\beta\frac{\rho^{\beta-1}}{2}|\del_x\rho|^2\right)\right)=0.
\end{equation}
In general, if we define the \emph{generalized Fisher} information functional 
\begin{equation}\label{eq:ken}
\mathcal{F}[\rho] = \int \kappa(\rho) |\partial_x \rho|^2 \, dx,
\end{equation} 
for a general smooth $\kappa:(0,\infty)\to (0,\infty)$, the corresponding \emph{Wasserstein gradient flow}, see \cite{O01, DM}, is formally given by 
\begin{equation}\label{eq:gfg}
\partial_t\rho+\del_x\left(\rho\del_x\left(\del_x(\kappa(\rho)\del_x\rho)-\frac{1}{2}\kappa'(\rho)|\del_x\rho|^2\right)\right)=0.
\end{equation}
In particular, \eqref{eq:gfb} corresponds to \eqref{eq:gfg} when $\kappa(\rho)=\rho^{\beta}$.\\

In some contexts \eqref{eq:ken} is called the {\em Korteweg energy functional}, \cite{K01}, and the dynamics driven by this type of energy appears in a wide range of physical and mathematical contexts, including fluid dynamics, quantum mechanics, capillary fluids, and interface evolution. In particular, equation \eqref{eq:gfg} arises also in the relaxation limit of several equations of fluid mechanics, \cite{GLT17, ACCLS21,AMZ24}. For the physical background on Korteweg theory we refer to the seminal paper of Dunn and Serrin \cite{DS85}.\\ 

The main result of this work is the construction of global-in-time {\em non-negative} weak solutions to \eqref{eq:gfb} with  initial condition 
\begin{equation}\label{eq:gfgid}
\rho|_{t=0}=\rho_0\geq 0.
\end{equation}
We assume only mild regularity hypotheses on $\rho_0$ and we consider the range \(\beta > -3\). To make things more precise, our main theorem is the following. For the definition of weak solution we refer to Section \ref{sec:ws}.
\begin{thm}\label{teo:main}
Assume that $\beta>-3$.  Let $\rho^0\in L^{1}(\T)$ be such that $\rho^0\geq 0$ if $\beta>-2$, and $m^0:=\min_{x\in\T}\rho^0(x)>0$ if $-3<\beta\leq-2$. Assume also that 
\begin{equation}\label{eq:idprecise}
\begin{aligned}
&\del_x\left(\rho^0\right)^{\frac{\beta}{2}+1}\in L^2(\T),& &\left(\rho^0\right)^{\frac{\beta+3}{2}}\in L^{1}(\T), & &\mbox{ if }\beta\not=-2,&\\
&\del_x\ln\rho^0\in L^2(\T),& & & &\mbox{ if }\beta=-2.&
\end{aligned}
\end{equation}
Then, there exists a global-in-time non-negative weak solution of \eqref{eq:gfb}-\eqref{eq:gfgid}, in the sense of Definition \ref{def:ws}.
\end{thm} 

Equation \eqref{eq:gfb} includes two well-known and well-studied particular cases. When $\beta=-1$, \eqref{eq:gfb} reduces to the {\em Quantum-Drift-Diffusion} equation 
\begin{equation}\label{eq:qdd}
\partial_t\rho+\dx\left(\rho\dx\left(\frac{\dx^2\sqrt{\rho}}{\sqrt{\rho}}\right)\right)=0,
\end{equation}
introduced in \cite{DLSS1,DLSS2} in the study of interface fluctuations in a certain spin system and in the analysis of semiconductors, \cite{JP}. 
On the other hand, when $\beta=0$, \eqref{eq:gfb} reduces to the {\em thin-film equation}
\begin{equation}\label{eq:tf}
\partial_t\rho+\dx(\rho\dx^3\rho))=0,
\end{equation}
which is a model for the capillarity-driven evolution of a viscous thin film over a solid substrate \cite{M,ODB} and can also be seen as a lubrication approximation of the Hele–Shaw flow \cite{GO01}.\\

The analysis of the Cauchy problem for the equations \eqref{eq:qdd} and \eqref{eq:tf} is well-understood. In particular, concerning the Quantum-Drift-Diffusion equation, the global-in-time existence of non-negative weak solutions has been proved in \cite{JP} for the one-dimensional case, and independently and with different methods in \cite{JM} and \cite{ST} for the multidimensional case. In particular, in \cite{JM} the analysis is based on entropy methods and time discretization, while in \cite{ST}, existence is obtained by exploiting the gradient flow structure and the Jordan-Kinderlehrer-Otto (JKO) scheme, \cite{JKO}. We also mention the result in \cite{F} where, remarkably, the uniqueness of the weak solutions constructed in \cite{JM,ST} is proved. For the thin-film equation, we refer to the seminal paper \cite{BF} for the global-in-time existence of non-negative weak solutions in one space dimension and to \cite{MDPGG} for the multi-dimensional case.\\

Several results are available regarding the Thin Film equation and the Quantum-Drift-Diffusion equation, covering different aspects and proving different properties of the solutions of \eqref{eq:qdd} and \eqref{eq:tf}. Without the intention of providing an exhaustive list, we refer to \cite{LT13, LT17, JP01, GJT06, DPGG98, O98, GO01, GKO08, DM05, Be96, MS15} and references therein.\\ 

Concerning the analysis of the Cauchy problem \eqref{eq:gfb}-\eqref{eq:gfgid}, fewer results are available. The global existence of non-negative weak solutions for $\beta\in[-1,0]$ has been proved for the first time in \cite{MMS} using the (JKO) scheme. In the one dimensional case and always in the range $\beta\in[-1,0]$, a different proof of the global existence of non-negative weak solutions based on time discretization and entropy methods is given in \cite{Bu} . Later, the range has been extended to $\beta\in [-2,1]$ in \cite{X} by using energy methods and in \cite{LMZ} by using the (JKO) scheme. It has been conjectured in \cite{MMS} that the existence of positive weak solutions in the one-dimensional case for \eqref{eq:B} holds for any $\beta>-3$, which is indeed the range considered in Theorem \ref{teo:main}.

\subsection{Strategy of the proof} 
The proof of Theorem \ref{teo:main} is based on a standard compactness argument. Formally, equation \eqref{eq:gfb} possesses a variety of entropy functionals dissipating along its solutions. The first one, related to the gradient flow structure, is given by 
\begin{equation}\label{eq:energyintro}
\begin{aligned}
\frac{d}{dt}\int \rho^\beta|\del_x\rho|^2dx + \int \rho\left|\del_x\left(\del_x(\rho^\beta\del_x\rho)-\frac{1}{2}\beta\rho^{\beta-1}|\del_x\rho|^2\right)\right|^2dx = 0.
\end{aligned}
\end{equation}
The second one is a zero order entropy, and yields for $\beta\not=-1$ the identity
\begin{equation}\label{eq:bdentropyintro1}
\begin{aligned}
\frac{d}{dt}\int\frac{4\rho^{\frac{\beta+3}{2}}}{(\beta+3)(\beta+1)}dx + \frac{4}{(\beta+1)^2}\int \rho^{\frac{\beta+3}{2}}|\del_x^2\rho^{\frac{\beta+1}{2}}|^2dx=0.
\end{aligned}
\end{equation}
Note that the dissipation term in \eqref{eq:bdentropyintro1} provides some control on the second derivative of $\rho$. In particular, by using Lemma \ref{lem:inequality}, it holds that for $\theta = \frac{3\beta+5}{4}$ there exists $C>0$ such that 
\begin{equation}\label{eq:bdentropyintro2}
\begin{aligned}
\int|\del_x^2\rho^\theta|^2dx \leq C \int \rho^{\frac{\beta+3}{2}}|\del_x^2\rho^{\frac{\beta+1}{2}}|^2dx.
\end{aligned}
\end{equation}
The above inequality generalizes the inequality 
\begin{equation}\label{eq:jm}
\begin{aligned}
\int|\del_x^2\sqrt{\rho}|^2dx \leq C \int \rho|\del_x^2\ln{\rho}|^2dx
\end{aligned}
\end{equation}
proved in \cite{JM}, and it is a particular case of a more general inequality proved in \cite[Theorem 2.1]{GLF}.  As a matter of fact, as a consequence of the result in \cite{GLF}, there is a family of zero-order functionals that get dissipated by the solutions of \eqref{eq:gfb}. In this paper we consider only \eqref{eq:bdentropyintro1}. In particular, we note that \eqref{eq:bdentropyintro1} coincides with zero-order entropy considered in \cite{MMS} only in the limiting case $\beta = -1$, where the logarithm must be considered, and not in the other cases.\\

Unfortunately, \eqref{eq:energyintro} and \eqref{eq:bdentropyintro1} are valid only for solutions that remain bounded away from zero. Consequently, it is necessary to develop an approximating scheme that permits sufficiently regular solutions, which are also bounded away from zero. This would allow us to establish appropriate approximate versions of the bounds \eqref{eq:energyintro} and \eqref{eq:bdentropyintro1}, and then pass to the limit in the approximate system to recover \eqref{eq:gfb}. Note that the a priori estimates, in particular \eqref{eq:bdentropyintro1}, are highly dependent on the structure of the equation. This makes the construction of a suitable approximate system non-trivial.\\

In order to devise the approximate scheme, we start by noticing that since we work in one space dimension, controlling the gradient of negative powers of the density helps to control the regions where $\rho$ vanishes through the Sobolev embedding. 
Thus, a first step will be to approximate $\kappa(\rho)=\rho^{\beta}$ with some $\kappa_\e(\re)$ involving some negative power of the density. In particular, we consider the following system 
\begin{equation*}
\begin{cases}
\partial_t\re+\dx(\re\ue)=0\\
\re\,\ue=\re\del_x\left(\del_x(\kappa_\e(\re)\del_x\re)-\frac{1}{2}\kappa_\e'(\re)|\del_x\re|^2\right). 
\end{cases}
\end{equation*}
Clearly, a further regularization is needed. A natural choice is adding a higher-order Laplacian in the second equation, namely
\begin{equation}\label{eq:regintro}
\begin{cases}
\partial_t\re+\dx(\re\ue)=0\\
\delta_{\e}(-\dx^2)^{s}\ue+\re\,\ue=\re\del_x\left(\del_x(\kappa_\e(\re)\del_x\re)-\frac{1}{2}\kappa_\e'(\re)|\del_x\re|^2\right),
\end{cases}
\end{equation}
with $\delta_{\e}\to 0$ as $\e\to 0$. 
At this point, we want to prove some suitable approximate version of \eqref{eq:energyintro} and \eqref{eq:bdentropyintro1}. Formally, for any fixed $\e>0$, the gradient flow estimate \eqref{eq:energyintro} is easily obtained. On the other hand, \eqref{eq:bdentropyintro1} requires some care. First, we exploit a general identity concerning the first variation of the Korteweg energy functional, see Proposition \ref{prop:equikort}, which allows to generalize the zero order entropy to the case when $\kappa(\rho)$ is not a power. This identity has already proved crucial in the analysis of compressible Navier-Stokes with degenerate viscosity and the Navier-Stokes-Korteweg equations in \cite{AS1, BN, BVY}. We also refer to \cite{ABS} for further generalization. However, in trying to obtain a suitable approximate version of \eqref{eq:bdentropyintro1}, it appears that unless we take $s=1$, the term $\delta_{\e}(-\dx^2)^{s}\ue$ seems to prevent obtaining the desired estimate. Thus, we consider 
\begin{equation}\label{eq:regintro1}
\begin{cases}
\partial_t\re+\dx(\re\ue)=0\\
-\delta_{\e}\dx^2\ue+\re\,\ue=\re\del_x\left(\del_x(\kappa_\e(\re)\del_x\re)-\frac{1}{2}\kappa_\e'(\re)|\del_x\re|^2\right).
\end{cases}
\end{equation}
We note that in \eqref{eq:regintro1} an issue of loss of derivatives appears. Indeed, the right-hand side of the elliptic equation for $\ue$ contains three derivatives in $\re$, and the continuity equation always loses derivatives with respect to the regularity of $\ue$. To solve the issue of the loss of derivatives, we use a skew-adjoint structure already known in the theory of Euler-Korteweg equations (see \cite{BGDD}). Indeed, introducing 
\begin{equation*}
\begin{aligned}
&\A_{\e}:=\sqrt{\frac{\kappa_\e}{\re}}\dx\re,& 
&\tk:=\sqrt{\re\kappa_\e},&
&\bar{\mu}(\re):=\re^{-1},& 
\end{aligned}
\end{equation*}
system \eqref{eq:regintro1} can be re-written as 
\begin{equation}\label{eq:appa2i}
\begin{cases}
\partial_t\A_{\e}+\dx(\tk\dx\ue)=-\dx(\ue\A_{\e})\\
-\deps\bar{\mu}(\re)\ddx\ue+\ue-\dx(\tilde{\kappa}(\re)\dx\A_{\e})=\frac{1}{2}\dx(\A_{\e}\A_{\e}).
\end{cases}
\end{equation}
Then the loss of derivatives disappears because the higher-order derivative terms appear in \eqref{eq:appa2i} in a skew-symmetric fashion.  We conclude noting that extending Theorem \ref{teo:main} to the entire real line is likely possible, but considering a bounded interval with appropriate boundary conditions is not trivial. Indeed, the validity of \eqref{eq:bdentropyintro1} is not clear, see \cite[Theorem 2.2]{GLF}.\\

\subsection*{Notation}
We use standard notation. In particular, the space of periodic smooth functions compactly supported on $(0,T)\times\T$ and with values on $\R$ will be denoted by $C^{\infty}_c((0,T)\times\T;\R)$. We will denote with $L^{p}(\T)$ the standard Lebesgue spaces and with $\|\cdot\|_{L^p}$ their norm. The Sobolev space of functions with $k$ distributional derivatives in $L^{p}(\T)$ is $W^{k,p}(\T)$ and in the case $p=2$ we will write $H^{k}(\T)$. The spaces $W^{-k,p}(\T)$ and $H^{-k}(\T)$ denote the dual spaces of $W^{k,p'}(\T)$ and $H^{k}(\T)$ respectively, where $p'$ is the H\"older conjugate of $p$. Given a Banach space $X$, the space $C(0,T;X_w)$ is the space of continuous functions with values in the space $X$ endowed with the weak topology. Moreover, we use the classical Bochner spaces for time-dependent functions with values in $X$, namely $L^{p}(0,T;X)$, $W^{k,p}(0,T;X)$ and $W^{-k,p}(0,T;X)$. When $X=L^p(\Omega)$, $X=H^{k}(\T)$, or $X=W^{k,p}(\T)$ the norms of the spaces $L^{q}(0,T;X)$ are denoted by $\|\cdot\|_{L^{q}_{t}(L^{p}_{x})}$, $\|\cdot\|_{L^{q}_{t}(H^{k}_{x})}$, and $\|\cdot\|_{L^{q}_{t}(W^{k,p}_{x})}$, respectively.

\subsection*{Organization of the paper}
The paper is organized as follows. In Section \ref{sec:ws} we give the precise definition of weak solutions and explain the formal {\em a priori} estimates available for equation \eqref{eq:gfb}. In Section \ref{sec:iden} we prove some important properties about the first variation of the Korteweg energy, such as Proposition \ref{prop:equikort}. Finally, in Section \ref{sec:appsys} we study the approximate system \eqref{eq:regintro1} and in Section \ref{sec:proof} we give the proof of Theorem \ref{teo:main}.  

\subsection*{Acknowledgments}
During the completion of the present paper, the first author has been fully supported by PRIN2022 ``Classical equations of compressible fluids mechanics: existence and
properties of non-classical solutions''. The second author gratefully acknowledge the partial support by the Gruppo
Na\-zio\-na\-le per l’Analisi Matematica, la Probabilit\`a e le loro
Applicazioni (GNAMPA) of the Istituto Nazionale di Alta Matematica
(INdAM), by the PRIN2022 ``Classical equations of compressible fluids mechanics: existence and properties of non-classical solutions'' and by the PRIN2022-PNRR ``Some
mathematical approaches to climate change and its impacts.'' 

\subsection*{Author Contributions} All authors contributed equally to this paper.

\subsection*{Data Availability} No datasets were generated or analyzed during the current study.


\subsection*{Conflict of interest} The authors declare no competing interests.

\section{A priori estimates and weak solutions}\label{sec:ws}

The aim of this section is to define the notion of weak solutions to equation \eqref{eq:gfb}. Recall that we assume $\beta>-3$. In order to give a meaningful definition of weak solutions, we show the derivation of the formal {\em a priori} estimates available for \eqref{eq:gfg}, namely \eqref{eq:energyintro} and \eqref{eq:bdentropyintro1}

\subsection{Energy estimate.} 
Multiplying \eqref{eq:gfg} by $-(\dx(\rho^{\beta}\ddx\rho)-\beta\frac{\rho^{\beta-1}}{2}|\dx\rho|^2)$ and integrating over $\T$, we obtain:
\begin{equation*}
\frac{d}{dt}\int \rho^{\beta}\frac{|\del_x\rho|^2}{2}\,dx+\int\rho\left|\del_x\left(\dx(\rho^{\beta}\ddx\rho)-\beta\frac{\rho^{\beta-1}}{2}|\dx\rho|^2\right)\right|^2\,dx=0.
\end{equation*}
Then, we have that 
\begin{equation}\label{eq:energy-est}
\begin{aligned}
&\frac{d}{dt}\int\frac{2|\del_x\rho^{\frac{\beta}{2}+1}|^2}{(\beta+2)^2}\,dx\leq 0,\mbox{ for }\beta\not=-2\\
&\frac{d}{dt}\int|\del_x\ln\rho|^2\,dx\leq 0,\mbox{ for }\beta=-2,
\end{aligned}
\end{equation}
which, together with the conservation of mass, Sobolev embeddings, and the restriction $\beta>-3$ (see Lemma \ref{lem:51}), implies that 
\begin{equation}\label{eq:aplinf}
\rho\in L^{\infty}(0,T;L^{\infty}(\T)). 
\end{equation}
Moreover, if $-3<\beta\leq -2$ it also holds that 
\begin{equation}\label{eq:vacest}
\frac{1}{\rho}\in L^{\infty}(0,T;L^{\infty}(\T)). 
\end{equation}
\begin{remark}
We note that for the Quantum-Drift-Diffusion equation, corresponding to $\beta=-1$, it was proved in \cite{JM06} that 
\begin{equation*}
\int|\dx^{3}\sqrt{\rho}|^2\,dx\leq C \int\rho\left|\dx\frac{\ddx\sqrt{\rho}}{\sqrt{\rho}}\right|^2\,dx,
\end{equation*}
while, for the case of the thin film equation, corresponding to $\beta=0$, it was proved in \cite{Be96} that 
\begin{equation*}
\int|\dx^3\rho^{\frac{3}{2}}|^2\,dx\leq C \int\rho|\dx^3\rho|^2\,dx
\end{equation*}
Instead, it is not clear to us whether the dissipation term
\begin{equation*}
\int\rho\left|\del_x\left(\dx(\rho^{\beta}\ddx\rho)-\beta\frac{\rho^{\beta-1}}{2}|\dx\rho|^2\right)\right|^2\,dx
\end{equation*}
provides some Sobolev control on some powers of the density without restricting the range of $\beta$. 
\end{remark}
\subsection{Entropy estimate.} 
Equation \eqref{eq:gfb} has an additional dissipating functional, which we shall call entropy. Considering only the cases $\beta\not=-1$ we note that by a direct calculation 
\begin{equation}\label{eq:idbeta}
\rho\dx\left(\dx(\rho^{\beta}\ddx\rho)-\beta\frac{\rho^{\beta-1}}{2}|\dx\rho|^2\right)=
\frac{2}{\beta+1}\dx\left(\rho^{\frac{\beta+3}{2}}\ddx\rho^{\frac{\beta+1}{2}}\right).
\end{equation}
Then, multiplying \eqref{eq:gfb} by $\frac{2}{\beta+1}\rho^{\frac{\beta+1}{2}}$ we obtain \eqref{eq:bdentropyintro1}, namely
\begin{equation}\label{eq:entropy}
    \frac{d}{dt}\int\frac{4\rho^{\frac{\beta+3}{2}}}{(\beta+3)(\beta+1)}dx + \frac{4}{(\beta+1)^2}\int \rho^{\frac{\beta+3}{2}}|\del_x^2\rho^{\frac{\beta+1}{2}}|^2dx=0.
\end{equation}
Using Lemma \ref{lem:inequality}, \eqref{eq:entropy} implies for $\beta\not=-5/3$ the {\em a priori} bounds 
\begin{equation}
\begin{aligned}\label{eq:theta}
&\rho^{\theta}\in L^{2}(0,T;H^{2}(\T))&
&\del_x\rho^{\frac{\theta}{2}}\in L^{4}((0,T)\times\T),
\end{aligned}
\end{equation}
where $\theta=\frac{3\beta+5}{4}$, and similarly for $\beta=-5/3$
\begin{equation}
\begin{aligned}\label{eq:theta1}
&\ln\rho\in L^{2}(0,T;H^{2}(\T))&
&\del_x\ln\rho\in L^{4}((0,T)\times\T).
\end{aligned}
\end{equation}
We also remark that in the limiting case $\beta=-1$ it holds that 
\begin{equation*}
\frac{d}{dt}\int\rho\ln\rho+\int\rho|\dx^2\ln\rho|^2\,dx=0.
\end{equation*}
\subsection{Definition of weak solutions.} The previous discussion motivates the following definition of weak solutions to our problem.
\begin{definition}\label{def:ws}
Let $\beta>-3$ and $\beta\not=-5/3$. A function $\rho$ is a weak solution of \eqref{eq:gfg} provided 
\begin{enumerate}
\item $\rho\geq 0$ with the regularity 
\begin{equation*}
\begin{aligned}
&\rho\in L^{\infty}(0,T;L^\infty(\T)),& &\rho^{\theta}\in L^{2}(0,T;H^{2}(\T)),& &\del_x\rho^{\frac{\theta}{2}}\in L^{4}((0,T)\times\T),&
\end{aligned}
\end{equation*}
\item for any $\psi\in C_{c}^{\infty}([0,T)\times\T)$, if $\beta\not=-\frac{5}{3}$ it holds that 
\begin{equation}\label{eq:ws}
\begin{aligned}
&-\iint\rho\partial_t\psi\,dxdt+\frac{1}{\theta}\iint\rho^{\beta+2-\theta}\del_x^2\rho^{\theta}\del_x^2 \psi\,dxdt\\
&\phantom{xx}-\frac{(\beta+3)}{\theta^2}\iint\rho^{\beta+2-\theta}|\del_x\rho^{\frac{\theta}{2}}|^2\del_x^2\psi\,dxdt -\int\rho^0\psi(0)\,dx=0,\\
\end{aligned}
\end{equation}
and if $\beta=-5/3$ it holds that 
\begin{equation}\label{eq:wsbis}
\begin{aligned}
&-\iint\rho\partial_t\psi\,dxdt+\iint\rho^{\beta+2}\del_x^2\ln\rho\,\del_x^2 \psi\,dxdt\\
&\phantom{xx}-\frac{(\beta+3)}{4}\iint\rho^{\beta+2}|\del_x\ln\rho|^2\del_x^2\psi\,dxdt -\int\rho^0\psi(0)\,dx=0,\\
\end{aligned}
\end{equation}
\item for almost every $t\in(0,T)$, if $\beta\not=-2$ it holds that 
\begin{equation*}
\int\,|\dx\rho(t,x)^\frac{\beta+2}{2}|^2\,dx\leq \int\,|\dx\rho^0(t,x)^\frac{\beta+2}{2}|^2\,dx
\end{equation*}
and if $\beta=-2$ it holds that 
\begin{equation*}
\int\,|\dx\ln\rho(t,x)|^2\,dx\leq \int\,|\dx\ln\rho^0(t,x)|^2\,dx.
\end{equation*}
\end{enumerate}
\end{definition}
\begin{remark}\mbox{}
If $\beta>-3$, then $\beta+2-\theta>0$. Thus, using \eqref{eq:aplinf} and \eqref{eq:theta}, every term in \eqref{eq:ws} is well-defined. We also note that in the range $-3<\beta\leq-5/3$, we have that $\theta\leq0$ and so \eqref{eq:vacest} holds.
\end{remark}

\section{Some preliminary remarks on the Korteweg energy}\label{sec:iden}

In this section we derive an identity that will prove useful in treating the Korteweg term. Given a smooth, strictly increasing function $\phi:(0,\infty)\mapsto \R$, we define the function $\mu:[0,\infty)\mapsto [0,\infty)$ by imposing $\mu'(\rho)=\rho\phi'(\rho)$ on $(0,\infty)$ and $\mu(0)=0$. We also define $\kappa:(0,\infty)\mapsto(0,\infty)$ as 
\begin{equation}\label{eq:mainrel}
\kappa(\rho):=\frac{(\mu'(\rho))^2}{\rho}. 
\end{equation}
The following proposition generalizes identity \eqref{eq:idbeta} and plays an important role in the construction of a suitable approximation scheme. In particular, it has been proved already in \cite{AS1} and in \cite{BN} and it can be generalized also in several dimensions.  
\begin{prop}\label{prop:equikort}
Assume \eqref{eq:mainrel} holds, then for any smooth $\rho>0$ 
\begin{equation}\label{eq:mainrelint1}
\begin{aligned}
\rho\del_x\left(\del_x(\kappa(\rho)\del_x\rho)-\frac{\kappa'(\rho)}{2}|\del_x\rho|^2\right)&=\partial_x(\rho\mu'(\rho)\partial_x^2\phi(\rho)).
\end{aligned}
\end{equation}
Moreover, 
\begin{equation}\label{eq:mainrelint}
\int\left(\del_x(\kappa(\rho)\del_x\rho)-\frac{\kappa'(\rho)}{2}|\del_x\rho|^2\right)\del_x^2\mu(\rho)\,dx=\int\rho\mu'(\rho)|\del_x^2\phi(\rho)|^2\,dx.
\end{equation}
\end{prop}
\begin{proof}
We first note that from \eqref{eq:mainrel} 
\begin{equation}\label{eq:derk}
\kappa'(\rho)=\frac{2\mu'(\rho)\mu''(\rho)}{\rho}-\frac{(\mu'(\rho))^2}{\rho^2}.
\end{equation}
Using \eqref{eq:derk} and the relation $\rho\phi'(\rho)=\mu'(\rho)$ we have 
\begin{equation*}
\begin{aligned}
\rho\del_x\left(\del_x(\kappa(\rho)\del_x\rho)-\frac{\kappa'(\rho)}{2}|\del_x\rho|^2\right)
&=\rho\del_x\left(\mu'(\rho)\del_x^2\phi(\rho)+\frac{(\phi'(\rho))^2}{2}|\del_x\rho|^2\right)\\
&=\del_x(\rho\mu'(\rho)\del_x^2\phi(\rho))-\del_x\rho\mu'(\rho)\del_x^2\phi(\rho)\\
& \phantom{xx} +\rho\del_x\frac{|\del_x\phi(\rho)|^2}{2}\\
&=\del_x(\rho\mu'(\rho)\del_x^2\phi(\rho))-\rho\del_x\phi(\rho)\del_x^2\phi(\rho)\\
& \phantom{xx} +\rho\del_x\frac{|\del_x\phi(\rho)|^2}{2}=\del_x(\rho\mu'(\rho)\del_x^2\phi(\rho)).
\end{aligned}
\end{equation*}
Thus, \eqref{eq:mainrelint1} is proved. Concerning \eqref{eq:mainrelint}, integrating by parts and using \eqref{eq:mainrelint1}, we have that 
\begin{equation*}
\begin{aligned}
\int\left(\del_x(\kappa(\rho)\del_x\rho)-\frac{\kappa'(\rho)}{2}|\del_x\rho|^2\right)\del_x^2\mu(\rho)\,dx&=-\int\del_x\left[\del_x(\kappa(\rho)\del_x\rho)-\frac{\kappa'(\rho)}{2}|\del_x\rho|^2\right]\rho\phi'(\rho)\del_x\rho\,dx\\
&=-\int\partial_x(\rho\mu'(\rho)\partial_x^2\phi(\rho)) \del_x\phi(\rho)\,dx\\&=\int\rho\mu'(\rho)|\del_x^2\phi(\rho)|^2\,dx. 
\end{aligned}
\end{equation*}
\end{proof}
The following technical lemma in combination with Proposition \ref{prop:equikort} will be useful to obtain certain {\em a priori} bounds. 
\begin{lemma}\label{lem:estkor}
Assume, in addition to the previous hypotheses, that there exists $C>0$ such that 
\begin{align}
&\rho|\mu''(\rho)|\leq C\mu'(\rho).\label{eq:power2}
\end{align}
Then, there exists $\tilde{C}>0$ such that 
\begin{equation}\label{eq:kort1}
\int\rho^2(\phi'(\rho))^3\left(|\del_x^2\rho|^2+\frac{|\del_x\rho|^4}{\rho^2}\right)\,dx\leq \tilde{C}\int\rho\mu'(\rho)|\del_x^2\phi(\rho)|^2\,dx.
\end{equation}
\end{lemma}
\begin{remark}\label{rem:sd}
Note that if for $k\in\R$ it holds that
\begin{equation*}
\int\rho^k\left(|\del_x^2\rho|^2+\frac{|\del_x\rho|^4}{\rho^2}\right)\,dx\leq C
\end{equation*}
then, by a direct calculation, it is easy to verify that 
\begin{equation*}
\int|\del_x^2\rho^{\theta}|^{2}+|\del_x\rho^\frac{\theta}{2}|^4\,dx\leq C
\end{equation*}
with $\theta=\frac{k+2}{2}$ if $k\not=-2$ and 
\begin{equation*}
\int|\del_x^2\ln\rho|^{2}+|\del_x\ln\rho|^4\,dx\leq C
\end{equation*}
if $k=-2$. 
\end{remark}
\begin{proof}
Using that $\rho\phi'(\rho)=\mu'(\rho)$ and expanding the square we have 
\begin{equation*}
\begin{aligned}
&\int\frac{\mu'(\rho)}{\rho}|\del_x(\mu'(\rho)\del_x\rho)|^2\,dx+\int\frac{(\mu'(\rho))^3|\del_x\rho|^4}{\rho^3}\,dx\\
=&\int\rho\mu'(\rho)|\del_x^2\phi(\rho)|^2\,dx+2\int\del_x(\mu'(\rho)\del_x\rho)\frac{\mu'(\rho)\del_x\rho}{\rho}\frac{\mu'(\rho)\del_x\rho}{\rho}\,dx\\
=&\int\rho\mu'(\rho)|\del_x^2\phi(\rho)|^2\,dx-4\int\frac{(\mu'(\rho))^2|\del_x\rho|^2}{\rho}\del_x^2\phi(\rho)\,dx,
\end{aligned}
\end{equation*}
where in the last line we integrated by parts and used again that $\rho\phi'(\rho)=\mu'(\rho)$. Then, by Young's inequality we obtain that there exists $C>0$ such that 
\begin{equation}\label{eq:ine1}
\int\frac{\mu'(\rho)}{\rho}|\del_x(\mu'(\rho)\del_x\rho)|^2\,dx+\int\frac{(\mu'(\rho))^3|\del_x\rho|^4}{\rho^3}\,dx\leq C\int\rho\mu'(\rho)|\del_x^2\phi(\rho)|^2\,dx. 
\end{equation}
Next, expanding again the square we have 
\begin{equation*}
\begin{aligned}
&\int\frac{(\mu'(\rho))^3}{\rho}|\del_x^2\rho|^2\,dx+\int\frac{\mu'(\rho)(\mu''(\rho))^2}{\rho}|\del_x\rho|^4\,dx\\
=&\int\frac{\mu'(\rho)}{\rho}|\del_x(\mu'(\rho)\del_x\rho)|^2\,dx-2\int\frac{(\mu'(\rho))^2\del^2_x\rho}{\rho}\mu''(\rho)|\del_x\rho|^2\,dx\\
\leq&\int\frac{\mu'(\rho)}{\rho}|\del_x(\mu'(\rho)\del_x\rho)|^2\,dx+2\int\frac{(\mu'(\rho))^{\frac{3}{2}}\del_x^2\rho}{\rho^{\frac{1}{2}}}
\frac{(\mu'(\rho))^{\frac{1}{2}}\rho|\mu''(\rho)||\del_x\rho|^2}{\rho^{\frac{3}{2}}}\,dx.
\end{aligned}
\end{equation*}
Using that $\rho|\mu''(\rho)|\leq C\mu'(\rho)$ and Young's inequality, we have that 
\begin{equation*}
\int\frac{(\mu'(\rho))^3}{\rho}|\del_x^2\rho|^2\,dx\leq C\int\frac{\mu'(\rho)}{\rho}|\del_x(\mu'(\rho)\del_x\rho)|^2\,dx
+C\int\int\frac{(\mu'(\rho))^3|\del_x\rho|^4}{\rho^3}\,dx. 
\end{equation*}
Then \eqref{eq:kort1} follows, after using \eqref{eq:ine1} and the relation $\rho\phi'(\rho)=\mu'(\rho)$.  
\end{proof}

\section{The approximating system}\label{sec:appsys}
In this section we study the approximating system. Let $0<\e<1$, we introduce the functions
\begin{equation*}
\begin{aligned}
&\mu_\e(\rho_\e) = \frac{2}{\beta+3} \rho_\e^{\frac{\beta+3}{2}} + \e \sqrt{\rho_\e},&
&\phie(\re)=\frac{2}{\beta+1}\re^{\frac{\beta+1}{2}}-\e\re^{-\frac{1}{2}},\\ &\Fe(\re)=\frac{4}{(\beta+1)(\beta+3)}\re^{\frac{\beta+3}{2}}-2\e\re^{\frac{1}{2}}+\frac{3}{2},& &\kappa_\e(\re)=\re^{\beta}+2\e\re^{\frac{\beta-2}{2}}+\e^2\re^{-2}.
\end{aligned}
\end{equation*}
By a direct computation, the functions above satisfy
\begin{equation*}
\begin{aligned}
&\kappa_\e(\re)=\frac{(\me'(\re))^2}{\re},& &\re\phie'(\re)=\me'(\re),& &\Fe'(\re)=\phie(\re).
\end{aligned}
\end{equation*}
To approximate equation \eqref{eq:gfb}, we consider on $(0,T)\times\T$ the system
\begin{equation}\label{eq:app}
\begin{aligned}
\del_t \rho_\e + \del_x( \rho_\e u_{\e})&= 0,\\
-\delta_{\e}\del_x^{2}u_{\e}+\re\,u_{\e}&=\del_x(\re\mu'(\re)\del_x^2\phie(\re)),
\end{aligned}
\end{equation}
where $\delta_{\e}=\e^{6}e^{-\frac{1}{2\e^2}}$, and we impose the initial datum 
\begin{equation}\label{eq:appid}
\re|_{t=0}=\re^0.
\end{equation}
The following theorem is the main result of this section. 
\begin{thm}\label{teo:exeps}
Let $k>3$ be a natural number and suppose that $\re^0\in H^{k+1}(\T)$ and assume that $\re^0(x)\geq m^0_{\e}:=\min_{x\in\T}\re^0(x)>0$. Then, 
\begin{enumerate}
\item There exists a unique global smooth solution $(\re,u_{\e})$ of the system \eqref{eq:app} such that 
\begin{equation}\label{eq:M}
(\re, u_{\e})\in C([0,T);H^{k+1}(\T))\times L^{2}(0,T;H^{k+1}(\T)).
\end{equation}
\item For any $t\in(0,T)$ 
\begin{equation}\label{eq:energy}
\begin{aligned}
\int\kappa_\e(\re(t))\frac{|\del_x\re(t)|^2}{2}\,dx&+\int_0^t\int\re|u_{\e}|^2\,dxds
+\delta_{\e}\int_0^t\int|\del_x u_{\e}|^2\,dxds\\
&=\int\kappa_\e(\re^0)\frac{|\del_x\re^0|^2}{2}\,dx.
\end{aligned}
\end{equation}
\item There exists a constant $K>0$ independent of $\e$ such that for any $t\in(0,T)$
\begin{equation}\label{eq:lvaceps}
e^{-\frac{K}{\e}}\leq \re(t,x)\leq e^{\frac{K}{\e}}. 
\end{equation}
\item There exists a constant $C>0$ independent of $\e$ such that for any $t\in(0,T)$
\begin{equation}\label{eq:bd}
\begin{aligned}
\int\,\Fe(\re(t))\,dx&+(1-\e)\iint \re(t)\me'(\re(t))|\del_x^2\phie(\re(t))|^2\,dsdx\\
&\leq \int\,\Fe(\re^0)\,dx+C\,\e
\end{aligned}
\end{equation}
\end{enumerate}
\end{thm}

\begin{proof} \mbox{}
\bigskip

We divide the proof into several steps. The notation $\lesssim$ means that the inequality is true up to a multiplicative constant.
\bigskip

\noindent\emph{Step 1: Local existence.}\\

 Let $k>3$ be a natural number. Without loss of generality, we can assume that $k$ is even. We start by regularizing the second equation of \eqref{eq:app} with a higher-order derivative. For $\lambda>0$ we consider 
\begin{equation}\label{eq:app_reg}
\begin{aligned}
\del_t \rho^{\lambda}_\e + \del_x( \rho^{\lambda}_\e u_{\e}^{\lambda})&= 0,\\
\lambda \del_x^{2k+2}\ue^{\lambda}-\delta_{\e}\del_x^{2}\ue^{\lambda}+\re^{\lambda}\,u_{\e}^{\lambda}&=\del_x(\re\mu'(\re^{\lambda})\del_x^2\phie(\re^{\lambda})),
\end{aligned}
\end{equation}
with initial datum \eqref{eq:appid}. We prove a local in time existence theorem of smooth solutions of \eqref{eq:app_reg} by means of a fixed point argument. To avoid heavy notation, we drop the indices $\epsilon$ and $\lambda$ for $\re^{\lambda}$ and $\ue^{\lambda}$. First, we consider the problem
\begin{equation}\label{eq:cont}
\begin{aligned}
& \del_t \rho + \del_x( \rho v) = 0, \\
& \rho|_{t=0} = \rho_0 \in H^{k+1}(\T)
\end{aligned}
\end{equation}
where $v \in B_1$, the unit ball in $L^2(0,T;H^{k+1}(\T))$.  By standard arguments, \eqref{eq:cont} admits a unique solution $\rho \in C([0,\infty), H^k(\T))$, which satisfies the bounds $m_0 e^{-T} \leq \rho(t, x) \leq M_0 e^T$, where $m_0 = \min \rho_0$ and $M_0 = \max \rho_0$. Note that since the continuity equation loses one derivative with respect to the regularity of $u$, we only obtain the $H^{k}$ spatial regularity for $\rho$. By employing the Kato-Ponce commutator estimate, see \cite{KP}, we can obtain the higher order estimate
\begin{equation}\label{eq:hkrhol}
\begin{split}
    \sup_t \|\rho\|_{H^{k}_x}^2 & \leq \|\rho_0\|_{H^{k}_x}^2 + C \int_0^T \|\rho\|_{H^{k}_x}^2\|v\|_{H^{k+1}_{x}}dt \\
    & \leq \|\rho_0\|_{H^{k}_x}^2 + CT^{1/2} \sup_t \|\rho\|_{H^{k}_x}^2.
\end{split}
\end{equation}
Choosing $T>0$ small enough, such that $CT^{1/2} \leq 1/2$, we find

\begin{equation} \label{eq:est}
    \sup_t \|\rho\|_{H^{k}_x}^2 \leq 2\|\rho_0\|_{H^{k}_x}^2.
\end{equation} 
Next, we consider the second equation of \eqref{eq:app_reg}, where $\rho$ is the one obtained after solving \eqref{eq:cont}. This is also a linear problem and can be solved by standard methods. Multiplying by $u$ and integrating by parts gives
\[ \lambda \int_0^T \| \del_x^{k+1} u \|_{L^{2}_x}^2 dt + \deps \int_0^T \|\del_xu\|_{L^{2}_x}^2 dt + \int_0^T \int_\T \rho|u|^2 dxdt \leq \int_0^T\int_T \rho \mu'(\rho)|\del_x^2\phi(\rho)||\del_xu| dxdt \]
and the right-hand side can be estimated as
\[ \begin{aligned}
    \int_0^T\int_T \rho \mu'(\rho)|\del_x^2\phi(\rho)||\del_xu| dxdt & \leq \int_0^T \|\rho \mu'(\rho)\del_x^2\phi(\rho)\|_{L^{2}_x} \|\del_xu\|_{L^{2}_x} dt \\
    & \leq C(\|\rho\|_{L^{\infty}_{x}}, \|\rho^{-1}\|_{L^{\infty}_{x}})\int_0^T \|\rho\|_{H^2_x} \|\del_xu\|_{L^{2}_x} dt \\
    & \leq C \int_0^T \|\rho\|_{H^{k}_x} \|u\|_{H^{k+1}_{x}} dt \\
    & \leq \frac{C}{\lambda}T \sup_t\|\rho\|^2_{H^k_x} + \frac{\lambda}{2}\int_0^T \|u\|_{H^{k+1}_{x}}^2dt.
\end{aligned} \]
By using \eqref{eq:est} and choosing $T=T(\lambda,\rho_0)$ small enough, we ensure that for $v \in B_1$, $u$ lies in $B_1$ as well. In other words, the mapping $v \mapsto u=S[v]$ from $L^2(0,T;H^{k+1}(\T))$ to $L^2(0,T;H^{k+1}(\T))$ satisfies $S[B_1] \subset B_1$. It remains to show that $S$ is a contraction. Let $v_1, v_2 \in B_1$. Similarly to the estimate above
\[ \begin{split}
    \sup_t\|\rho_1 - \rho_2\|_{H^{k}_x}^2 & \leq \sup_t \|\rho_1 - \rho_2\|_{H^{k}_x}^2 \int_0^T \|v_1\|_{H^{k+1}_{x}} dt + \sup_t\|\rho_2\|_{H^{k}_x} \int_0^T \|v_1 - v_2\|_{H^{k+1}_{x}} dt \\
    & \phantom{xx} + \sup_t\|\rho_1\|_{H^{k}_x} \int_0^T \|v_1 - v_2\|_{H^{k+1}_{x}} dt + \sup_t \|\rho_1 - \rho_2\|_{H^{k}_x}^2 \int_0^T \|v_2\|_{H^{k+1}_{x}} dt \\
    & \leq CT^{1/2} \sup_t \|\rho_1 - \rho_2\|_{H^{k}_x}^2 + C \|\rho_0\|_{H^{k}_x}^2 T^{1/2} \|v_1-v_2\|_{L^2_t(H^{k+1}_{x})}.
\end{split} \]
and choosing $T$ such that $CT^{1/2} \leq 1/2$, we obtain
\begin{equation}\label{aux}
    \sup_t\|\rho_1 - \rho_2\|_{H^{k}_x}^2 \leq CT^{1/2} \|v_1-v_2\|_{L_t^2(H^{k+1}_{x})}
\end{equation}
Now, given $\rho_1, \rho_2$, we consider the equation for the difference of the velocities, which gives after multiplying by $u_1-u_2$ and integrating over $(0,T)\times\T$:
\[ \begin{split}
    \lambda \int_0^T \|u_1 - u_2\|_{H^{k+1}_{x}}^2 dt + \int_0^T\int_\T \rho_1|u_1-u_2|^2dxdt & \leq C \int_0^T\|\rho_1-\rho_2\|_{H^2_x} \|u_1-u_2\|_{H^1_x}dt \\
    & \phantom{xx} + \int_0^T\int_\T(\rho_1-\rho_2)u_2(u_1-u_2)dxdt:= I + II
\end{split} \]
and the constant $C>0$ depends on the $L^\infty$ norms of $\rho_1, \rho_2, \rho_1^{-1}, \rho_2^{-1}$ and on $\|\rho_0\|_{H^k_x}$. We start by estimating $II$ as follows:
\[ \begin{split}
    II & \leq \int_0^T\int_\T |\rho_1-\rho_2||u_2||u_1-u_2| dxdt \\
    & \leq e^{\frac{CT}{2}} \int_0^T\int_\T |\rho_1-\rho_2||u_2|\rho_1^{1/2}|u_1-u_2|dxdt \\
    & \leq e^{CT} \int_0^T \|\rho_1-\rho_2\|_{L^{\infty}_{x}} \|u_2\|_{L^{2}_x} \|\rho_1^{1/2}(u_1-u_2)\|_{L^{2}_x}^2 dt \\
    & \leq C \sup_t\|\rho_1-\rho_2\|_{H^{k}_x}^2 \int_0^T \|u_2\|_{H^{k+1}_{x}}^2dt + \frac{1}{2} \int_0^T\int_\T \rho_1|u_1-u_2|^2dxdt
\end{split} \]
and using \eqref{aux} and the fact that $u_2 \in B_1$, we find
\[ II \leq CT \|v_1-v_2\|_{L^2(H^{k+1}_{x})}^2 + \frac{1}{2} \int_0^T\int_\T \rho_1|u_1-u_2|^2dxdt \]
Concerning $I$, we have:
\[ \begin{split}
    I & \leq CT^{1/2} \sup_t\|\rho_1 - \rho_2\|_{H^{k}_x} \|u_1-u_2\|_{L^2_t(H^{k+1}_{x})} \\
    & \leq CT \|v_1-v_2\|_{L^2_t(H^{k+1}_{x})}\|u_1-u_2\|_{L^2_t(H^{k+1}_{x})}
\end{split} \]
and putting everything together, we arrive at
\[ \|u_1-u_2\|_{L^2_t(H^{k+1}_{x})} \leq C(\lambda,\|\rho_0\|_{H^{k}_x}) T \|v_1-v_2\|_{L^2_t(H^{k+1}_{x})} \]

\noindent and choosing $T$ small enough, we conclude that $S$ is a contraction. An application of Banach's fixed point theorem, gives the existence of a unique strong solution to \eqref{eq:app_reg} for a small time. Precisely, going back to the notation $(\re^{\lambda}, \ue^{\lambda})$, we have proved that given $\lambda>0$ and $\e>0$, there exists a time $T^{*}_{\lambda,\e}$ and a pair 
\begin{equation}\label{eq:41}
(\re^{\lambda}, u^{\lambda}_{\e})\in C([0,T_{\lambda,\e}^{*});H^{k}(\T))\times L^2(0,T_{\lambda,\e}^{*};H^{k+1}(\T))
\end{equation}
solving \eqref{eq:app} in the classical sense.\\

\noindent\emph{Step 2: Energy estimate.}\\

Multiplying by $-\left(\dx(\kappa_\e(\re^{\lambda})\dx\re^{\lambda})-\frac{\kappa_\e'(\re^{\lambda})}{2}|\dx\re^{\lambda}|^2\right)$ the first equation in \eqref{eq:app_reg} and integrating over $\T$, we obtain
\begin{equation*}
\begin{aligned}
&\int\partial_t\re^{\lambda}\left(\partial_x(\kappa_\e(\re^{\lambda})\dx\re^{\lambda})-\frac{\kappa_\e'(\re^{\lambda})}{2}|\dx\re^{\lambda}|^2\right)\,dx\\
 = &\int\ue^{\lambda}\re^{\lambda}\dx\left(\partial_x(\kappa_\e(\re^{\lambda})\dx\re^{\lambda})-\frac{\kappa_\e'(\re^{\lambda})}{2}|\dx\re^{\lambda}|^2\right)\,dx.
\end{aligned}
\end{equation*}
Manipulating the first term, we get 
\begin{equation}\label{eq:42}
\frac{d}{dt}\int\kappa_\e(\re^{\lambda})\frac{|\dx\re^{\lambda}|^2}{2}\,dx+\int\ue^{\lambda}\re^{\lambda}\dx\left(\partial_x(\kappa_\e(\re^{\lambda})\dx\re^{\lambda})-\frac{\kappa_\e'(\re^{\lambda})}{2}|\dx\re^{\lambda}|^2\right)\,dx=0.
\end{equation}
On the other hand, multiplying the second equation by $\ue$ we obtain 
\begin{equation}\label{eq:43}
\begin{aligned}
\lambda&\int|\del_x^{k+1}u^{\lambda}_\e|^2dx+\deps\int|\dx\ue^{\lambda}|^2\,dx+\int\re^{\lambda}|u^{\lambda}_\e|^2\,dx\\
&-\int\ue^{\lambda}\re^{\lambda}\dx\left(\partial_x(\kappa_\e(\re^{\lambda})\dx\re^{\lambda})-\frac{\kappa_\e'(\re^{\lambda})}{2}|\dx\re^{\lambda}|^2\right)\,dx=0. 
\end{aligned}
\end{equation}
Then summing up \eqref{eq:42} and \eqref{eq:43} and integrating in time, we obtain for any $t\in[0,T^{*}_{\lambda,\e})$

\begin{equation} \label{eq:43-lambda}
    \begin{aligned}
    \int\kappa_\e(\re^{\lambda}(t))\frac{|\del_x\re^{\lambda}(t)|^2}{2}\,dx&+\int_0^t\int\re^{\lambda}|u^{\lambda}_{\e}|^2\,dxds+\delta_{\e}\int_0^t\int|\del_x u^{\lambda}_{\e}|^2\,dxds\\
    &+\lambda\int_0^t\int|\del_x^{k+1} u^{\lambda}_{\e}|^2\,dxds =\int\kappa_\e(\re^0)\frac{|\del_x\re^0|^2}{2}\,dx.
\end{aligned}
\end{equation}

\noindent\emph{Step 3: $L^\infty$-bounds.}\\

By the definition of $\kappa_\e$ and \eqref{eq:43-lambda} we deduce that there exists $K>0$ independent of $\e$ and $\lambda$ such that 
\begin{equation}\label{eq:lne}
\sup_{t\in[0,T^{*}_{\lambda,\e})}\e^2\int|\dx\ln\re^{\lambda}|^2\,dx\leq K^2.
\end{equation}
Since $\re$ has unit average, we also have that 
\begin{equation*}
|\ln\re(t,x)|\leq \int|\dx\ln\re^{\lambda}|\,dx,
\end{equation*}
and thus 
\begin{equation*}
\e|\ln(\re^{\lambda}(t,x))|\leq \left(\e^2\int|\dx\ln\re^{\lambda}|^2\,dx\right)^{\frac{1}{2}}\leq K. 
\end{equation*}
Hence, for any $(t,x)\in t\in[0,T^{*}_{\lambda,\e})\times\T$
\begin{equation}\label{bounds}
e^{-\frac{K}{\e}}\leq \re^{\lambda}(t,x)\leq e^{\frac{K}{\e}}.
\end{equation}

\noindent\emph{Step 4: Uniqueness for fixed $\lambda$.}\\

Let $(\re^{\lambda}, \ue^{\lambda})=(\rho, u)$ and $(\bar{\re}^{\lambda},\bar{\ue}^{\lambda})=(\bar{\rho}, \bar{u})$ be two solutions of \eqref{eq:app_reg} with the same initial datum. Then, a standard estimate on the difference of solutions of the continuity equation gives 
\begin{equation*}
\frac{d}{dt}\|\rho-\bar{	\rho}\|_{H^1_x}^2\lesssim \|\dx\bar{u}\|_{L^{\infty}_x}\|\rho-\bar{\rho}\|_{H^1_x}^2+\|u-\bar{u}\|_{H^{2}_x}\|\rho-\bar{\rho}\|_{H^1_x}. 
\end{equation*}
Recalling that by a direct calculation, using \eqref{eq:mainrel}, we have that 
\begin{equation*}
\dx(\rho\mu'(\rho)\dx^2\phi(\rho))=\dx\mathcal{K}(\rho,\dx\rho), 
\end{equation*}
with 
\begin{equation*}
\mathcal{K}(\rho,\dx\rho)=\rho\dx(\kappa(\rho)\dx\rho)-\frac{1}{2}\left(\kappa(\rho)|\dx\rho|^2-\rho\kappa’(\rho)\right)|\dx\rho|^2,
\end{equation*}
 by the second equation in \eqref{eq:app_reg} we obtain that 
\begin{equation*}
\begin{aligned}
\lambda\|\dxkpo(u-\bar{u})\|_{L^{2}_x}^2+\|u-\bar{u}\|_{L^{2}_x}^2&\lesssim \left|\int \mathcal{K}(\rho,\dx\rho)-\mathcal{K}(\bar{\rho},\dx\bar{\rho})(\dx u-\dx\bar{u})\,dx\right|\\
&+\int|\rho-\bar{\rho}||\bar{u}||u-\bar{u}|\,dx. 
\end{aligned}
\end{equation*}
Thus, by using \eqref{eq:41} and \eqref{bounds}, integrating by parts the first term, we deduce that 
\begin{equation*}
\frac{d}{dt}\|\rho-\bar{\rho}\|_{H^1_x}^2+\lambda\|\dxkpo(u-\bar{u})\|_{L^{2}_x}^2+\|u-\bar{u}\|_{L^{2}_x}^2\lesssim \|\rho-\bar{\rho}\|_{H^1_x}\|u-\bar{u}\|_{H^{2}_x}. 
\end{equation*}
By using Young's inequality and Gr\"onwall's Lemma we conclude.\\ 

\noindent\emph{Step 5: Global existence for fixed $\lambda$.}\\

We note that by \eqref{eq:43-lambda}, \eqref{bounds}, and the very same argument we used in order to prove \eqref{eq:hkrhol}, we can deduce that for any $t\in [0,T^{*}_{\lambda})$ there exists a constant $C$ depending on $\lambda$ and $\e$, but independent of $T^{*}_{\lambda,\e}$ such that 
\begin{equation}\label{eq:final C}
\sup_{t\in[0,T^{*}_{\lambda})}\|\re^{\lambda}\|_{H^{k}_x}^2+\int\|\ue^{\lambda}\|_{H^{k+1}_{x}}^2\,dt\leq C.
\end{equation}
Therefore, the solution $(\re^{\lambda},\ue^{\lambda})$ can be extended to the space $C([0,T);H^{k}(\T))\times L^2(0,T;H^{k+1}(\T))$, for any $T>0$.\\

\noindent\emph{Step 6: Estimates uniform in ${\lambda}$.}\\

As remarked in {\em Step 5}, the constant $C>0$ depends on $\lambda$. The objective of this step is to prove a $\lambda$-independent estimate in the space $C([0,T);H^{k}(\T))\times L^2(0,T;H^{k+1}(\T))$. Because of the higher-order regularization, we are only able to obtain this estimate for a small time $T=T^{*}>0$. As in {\em Step 1} to avoid heavy notation we omit the indices $\e$ and $\lambda$ for the pair $(\re^{\lambda},\ue^{\lambda})$. As explained in the Introduction, it is convenient to introduce the following quantities 
\begin{equation}\label{eq:changevar}
\begin{aligned}
&\A:=\sqrt{\frac{\kappa_\e(\rho)}{\rho}}\dx\rho,&
&\tk:=\sqrt{\rho\kappa_\e(\rho)},& &\bm:=\rho^{-1}.&\\
\end{aligned}
\end{equation}
Then, manipulating \eqref{eq:app} we deduce
\begin{equation}\label{eq:appa}
\begin{cases}
\partial_t\A+\dx(\tk\dx u)=-\dx(u\A)\\
\lambda\bar{\mu}_\varepsilon(\rho)\del_x^{2k+2} u-\deps\bm\ddx u+ u-\dx(\tk\dx\A)=\frac{1}{2}\dx(\A\A). 
\end{cases}
\end{equation}
Using \eqref{eq:43-lambda} and \eqref{bounds}, which because of {\em Step 4} hold for any $T>0$, we can infer that for any fixed $0<\e<1$ and uniformly in $\lambda$, holds that 
\begin{equation}\label{eq:412}
\begin{aligned}
&\A\in L^{\infty}(0,T;L^{2}(\T)),\,  u\in L^{2}(0,T;H^{1}(\T))\\
&\rho, \rho^{-1}\in L^{\infty}(0,T;L^{\infty}(\T)),\, \sqrt{\lambda}\del_x^{k+1} u \in L^2(0,T;L^2(\T)). 
\end{aligned}
\end{equation}
The goal is to derive an $H^{k}$ estimate that is independent of $\lambda$. By applying $\dx^{k}$ to both equations of \eqref{eq:appa} and multiplying the resulting equations by $\dxk \A$ and $\dxk u$, respectively, after integrating in space we obtain 
\begin{equation*}
\begin{aligned}
\frac{d}{dt}\frac{\|\dxk \A\|_{L^{2}_x}^{2}}{2}&+\|\dxk u\|_{L^{2}_x}^2-\lambda(\dxk(\bar{\mu}(\rho)\dx^{2k+2}u), \dxk u)-(\dxkpo(\tk\dx\A), \dxk u)\\
&+(\dxkpo(\tk\dx u), \dxk \A)-\deps(\dxk (\bm\dx^2 u), \dxk u)\\
&=-(\dxkpo(u\A), \dxk \A)+\frac{1}{2}(\dxkpo(\A\A), \dxk u). 
\end{aligned}
\end{equation*}
Integrating by parts and recalling that $k$ is even, we have that 
\begin{equation*}
\begin{aligned}
\frac{d}{dt}&\frac{\|\dxk \A\|_{L^{2}_x}^{2}}{2}+\|\dxk u\|_{L^{2}_x}^2+\deps(\bm\dxkpo u, \dxkpo u)+\lambda(\bar{\mu}(\rho)\dxkpo u,\dxkpo u)\\
&=\frac{1}{2}\lambda(\dx^2\bm\dx^{2k} u,\dx^{2k} u)-\deps([\dxk,\bm]\dx u, \dxkpo u)\\
&-\deps(\dxk(g( \rho)\A\dx u), \dxk u)-(\dx\tk\dxkpo u, \dxk \A)\\
&-([\dxkpo, \tk]\dx u, \dxk\A)-([\dxk,\tk]\dx\A, \dxkpo u)\\
&-(u\dx\dxk \A, \dxk \A)-([\dxk, u]\dx\A, \dxk \A)\\
&-(\dxk(\dx u\A), \dxk \A)-\frac{1}{2}(\dxk(\A\A), \dxkpo u),
\end{aligned}
\end{equation*}
where $[\dxk, f]h:=\dxk(f h)- f\dxk h$ and $g(\rho):=(\bar{\mu}’(\rho)\rho^{\frac{1}{2}})/([\tilde{\kappa}(\rho)]^{\frac{1}{2}})$. Using \eqref{eq:412}, the Kato-Ponce inequality and the Kato-Ponce commutator estimate, see \cite{KP}, we have that 
\begin{equation}\label{eq:m1}
\begin{aligned}
\partial_t&\|\dxk\A\|_{L^{2}_x}^2+\|\dxk u\|_{L^{2}_x}^2+\deps\|\dxkpo u\|_{L^{2}_x}^2+\lambda\|\dx^{2k+1}u\|_{L^{2}_x}^2\lesssim \lambda\|\dx^2\bm\|_{L^{\infty}_x}\|\dx^{2k}u\|_{L^{2}_x}^2\\
&+\deps\|\dxk \bm\|_{L^{2}_x}\|\dx u\|_{L^{\infty}_x}\|\dxkpo u\|_{L^{2}_x}+\deps\|\dx\bm\|_{L^{\infty}_x}\|\dxk u\|_{L^{2}_x}\|\dxkpo u\|_{L^{2}_x}\\
&+\deps\|\dxk g(\rho)\|_{L^{2}_x}\|\A\|_{L^{\infty}_x}\|\dx u\|_{L^{\infty}_x}\|\dxk u\|_{L^{2}_x}+\deps\|g(\rho)\|_{L^{\infty}_x}\|\dxk \A\|_{L^{2}_x}\|\dx u\|_{L^{\infty}_x}\|\dxk u\|_{L^{2}_x}\\
&+\deps\|g(\rho)\|_{L^{\infty}_x}\|A\|_{L^{\infty}_x}\|\dxkpo u\|_{L^{2}_x}\|\dxk u\|_{L^{2}_x}+\|\dx\tk\|_{L^{\infty}_x}\|\dxkpo u\|_{L^{2}_x}\|\dxk \A\|_{L^{2}_x}\\
&+\|\dxkpo\tk\|_{L^{2}_x}\|\dx u\|_{L^{\infty}_x}\|\dxk\A\|_{L^{2}_x}+\|\dx\tk\|_{L^{\infty}_x}\|\dxkpo u\|_{L^{2}_x}\|\dxk\A\|_{L^{2}_x}\\
&+\|\dxk\tk\|_{L^{2}_x}\|\dx\A\|_{L^{\infty}_x}\|\dxkpo u\|_{L^{2}_x}+\|\dx\tk\|_{L^{\infty}_x}\|\dxk\A\|_{L^{2}_x}\|\dxkpo u\|_{L^{2}_x}\\
&+\|\dx u\|_{L^{\infty}_x}\|\dxk\A\|_{L^{2}_x}^2+\|\A\|_{L^{\infty}_x}\|\dxkpo u\|_{L^{2}_x}\|\dxk\A\|_{L^{2}_x}+\|\dx\A\|_{L^{\infty}_x}\|\dxk u\|_{L^{2}_x}\|\dxk\A\|_{L^{2}_x}.
\end{aligned}
\end{equation}
We note that as a consequence of \eqref{eq:412} for any function $f\in C^{\infty}((0,\infty))$ we have that 
\begin{equation}\label{eq:414}
\begin{aligned}
& \|f( \rho)\|_{L^{\infty}_x}\lese 1,\quad\|\dx f( \rho)\|_{L^{\infty}_x}\lesssim \|\A\|_{L^{\infty}_x}\\
 &\|\dx^2 f( \rho)\|_{L^{\infty}_x}\lesssim \|\A\|_{L^{\infty}_x}^2+\|\dx\A\|_{L^{\infty}_x}& \\
& \|\dxk f( \rho)\|_{L^{2}_x}\lese \| \rho\|_{H^{k}_x},\quad\|\dxkpo f( \rho)\|_{L^{2}_x}\lesssim\|\dxk \A\|_{L^{2}_x}+\|\A\|_{L^{\infty}_x}\| \rho\|_{H^{k}_x}.
\end{aligned}
\end{equation}
Precisely, the implicit constants in \eqref{eq:414} depend on $\e$, on the initial data (through the bounds in \eqref{eq:412}), but not on $\lambda$. Thus, by using \eqref{eq:414} and Young's inequality, from \eqref{eq:m1} we obtain the following inequality 
\begin{equation}\label{eq:m2}
\begin{aligned}
\partial_t&\|\dxk\A\|_{L^{2}_x}^2+\|\dxk u\|_{L^{2}_x}^2+\deps\|\dxkpo u\|_{L^{2}_x}^2+\lambda\|\dx^{2k+1}u\|_{L^{2}_x}^2\\
&\lesssim \lambda\|\dx\A\|_{L^{\infty}_x}\|\dx^{2k}u\|_{L^{2}_x}^2+\lambda\|\A\|^2_{L^{\infty}_{x}}\|\dx^{2k} u\|_{L^{2}_x}^2+\deps\|\A\|^2_{L^{\infty}_{x}}\|\dxk u\|_{L^{2}_x}^2\\
&+(1+\|\A\|^2_{L^{\infty}_{x}}+\|\dx\A\|_{L^{\infty}_x}^2+\|\dx u\|_{L^{\infty}_x}^2+\|\A\|_{L^{\infty}_x}^2\|\dx u\|_{L^{\infty}_x}^2)(\|\dxk \A\|_{L^{2}_x}^2+\|\rho\|_{H^{k}_x}^2).
\end{aligned}
\end{equation}
We treat the last three terms on the right-hand side. By using the interpolation inequality 
\begin{equation}\label{eq:iterpol}
\|u\|_{H^{s}_x}\lesssim \|u\|_{H^{\tau}_x}^{1-\theta}\|u\|_{H_x^{l}}^{\theta},\qquad \theta=\frac{s-\tau}{l-\tau}
\end{equation}
which holds for any $s\in(\tau,l)$ we obtain, using Young's inequality, that
\begin{equation*}
\begin{aligned}
\deps\|\A\|_{L^{\infty}_x}^2\|\dxk u\|_{L^{2}_x}^2&\lesssim\deps\|\A\|_{L^{\infty}_x}^2\|\dx u\|_{L^{2}_x}^{\frac{2}{k}}\|\dxkpo u\|_{L^{2}_x}^{\frac{2(k-1)}{k}}\\
&\leq C\deps\|\A\|_{L^{\infty}_x}^{2k-2}\|\dxk \A\|_{L^{2}_x}^{2}+\tilde{c}\deps\|\dxkpo u\|_{L^{2}_x}^2,\\
\lambda\|\A\|^2_{L^{\infty}_{x}}\|\dx^{2k} u\|_{L^{2}_x}^2&\lesssim\lambda \|\A\|^2_{L^{\infty}_x}\|\dxkpo u\|_{L^{2}_x}^{\frac{2}{k}}\|\dx^{2k+1}u\|_{L^{2}_x}^{\frac{2(k-1)}{k}}\\
&\leq C\lambda \|\dxk\A\|_{2}^{2k}\|\dxkpo u\|_{L^{2}_x}^2+\tilde{c}\lambda\|\dx^{2k+1} u\|_{L^{2}_x}^2,\\
\lambda\|\dx\A\|_{L^{\infty}_x}\|\dx^{2k} u\|_{L^{2}_x}^2&\lesssim\lambda \|\dx\A\|_{L^{\infty}_x}\|\dxkpo u\|_{L^{2}_x}^{\frac{2}{k}}\|\dx^{2k+1}u\|_{L^{2}_x}^{\frac{2(k-1)}{k}}\\
&\leq C\lambda \|\dxk\A\|^{k}_{2}\|\dxkpo u\|_{L^{2}_x}^2+\tilde{c}\lambda\|\dx^{2k+1} u\|_{L^{2}_x}^2\\
\end{aligned}
\end{equation*}
where $\tilde{c}$ is small.
Therefore \eqref{eq:m2} implies
\begin{equation}\label{eq:m3}
\begin{aligned}
&\partial_t\|\dxk \A\|_{L^{2}_x}^2+\|\dxk u\|_{L^{2}_x}^2+\deps\|\dxkpo u\|_{L^{2}_x}^2+\lambda\|\dx^{2k+1} u\|_{L^{2}_x}^2\\
\lesssim&(1+\|\A\|_{L^{\infty}_x}^{2k-2}+\|A\|_{L^{\infty}_x}^{2}+\|\dx u\|_{L^{\infty}_x}^2+\|\dx \A\|_{L^{\infty}_x}^2+\|\A\|_{L^{\infty}_x}^2\|\dx u\|_{L^{2}_x}^2)(\|\dxk\A\|_{L^{2}_x}^2+\|\rho\|_{H^{k}_x}^2)\\
+&\lambda\|\dxkpo u\|_{L^{2}_x}^2\|\dxk\A\|_{L^{2}_x}^k+\lambda\|\dxkpo u\|_{L^{2}_x}^2\|\dxk\A\|_{L^{2}_x}^{2k}.
\end{aligned}
\end{equation}
By using the continuity equation as in the previous steps and using that $\re$ is bounded, we can deduce that 
\begin{equation*}
\frac{d}{dt}\|\rho\|_{H^{k}_x}^2\lesssim \|\dx u\|_{L^{\infty}_x}\|\rho\|_{H^{k}_x}^2+\|\dxkpo u\|_{L^{2}_x}\|\rho\|_{H^{k}_x},
\end{equation*}
which combined with \eqref{eq:m3} gives 
\begin{equation}\label{eq:m4}
\begin{aligned}
&\partial_t(\|\dxk \A\|_{L^{2}_x}^2+\|\rho\|_{H^{k}_x}^2)+\|\dxk u\|_{L^{2}_x}^2+\deps\|\dxkpo u\|_{L^{2}_x}^2+\lambda\|\dx^{2k+1} u\|_{L^{2}_x}^2\\
\lesssim&(1+\|\A\|_{L^{\infty}_x}^{2k-2}+\|\A\|_{L^{\infty}_x}^{2}+\|\dx u\|_{L^{\infty}_x}^2+\|\dx \A\|_{L^{\infty}_x}^2+\|\A\|_{L^{\infty}_x}^2\|\dx u\|_{L^{2}_x}^2)(\|\dxk\A\|_{L^{2}_x}^2+\|\rho\|_{H^{k}_x}^2)\\
+&\lambda\|\dxkpo u\|_{L^{2}_x}^2\|\dxk\A\|_{L^{2}_x}^{2k}+\lambda\|\dxkpo u\|_{L^{2}_x}^2\|\dxk\A\|_{L^{2}_x}^{2k}. 
\end{aligned}
\end{equation}
Although inequality \eqref{eq:m4} will be important in the following steps, it is still not enough to deduce a useful bound. However, using Sobolev embeddings we can infer that there exists $q>2$ such that 
\begin{equation*}
\begin{aligned}
&\partial_t(\|\dxk \A\|_{L^{2}_x}^2+\|\rho\|_{H^{k}_x}^2)+\|\dxk u\|_{L^{2}_x}^2+\deps\|\dxkpo u\|_{L^{2}_x}^2+\lambda\|\dx^{2k+1} u\|_{L^{2}_x}^2\\
\lesssim&(1+\lambda\|\dxkpo u\|_{L^{2}_x}^2+\|\dx^2 u\|_{L^{2}_x}^2)(\|\dxk \A\|_{L^{2}_x}^2+\|\rho\|_{H^{k}_x}^2)^{q}.
\end{aligned}
\end{equation*}
Note that by using \eqref{eq:iterpol} we have $\|\dx^2 u\|_{L^{2}_x}^2\leq \|\dx u\|_{L^{2}_x}^{\frac{2(k-1)}{k}}\|\dxkpo u\|_{L^{2}_x}^{\frac{2}{k}}$. Thus, by Young's inequality we obtain, after redefining $q>2$, that
\begin{equation}\label{eq:m5}
\begin{aligned}
&\partial_t(\|\dxk \A\|_{L^{2}_x}^2+\|\re\|_{H^{k}_x}^2)+\|\dxk u\|_{L^{2}_x}^2+\deps\|\dxkpo u\|_{L^{2}_x}^2+\lambda\|\dx^{2k+1} u\|_{L^{2}_x}^2\\
\lesssim&(1+\lambda\|\dxkpo u\|_{L^{2}_x}^2+\|\dx u\|_{L^{2}_x}^2)(\|\dxk \A\|_{L^{2}_x}^2+\|\re\|_{H^{k}_x}^2)^{q}.
\end{aligned}
\end{equation}
By using \eqref{eq:412}, we deduce from \eqref{eq:m5} that there exist $T^{*}>0$ and $C>0$, both independent of $\lambda$ such that 
\begin{equation*}
\sup_{t\in (0,T^{*})}(\|\dxk \A\|_{L^{2}_x}^2+\|\rho\|_{H^{k}_x}^2)+\int_0^{T^{*}}\|\dxkpo u\|_{L^{2}_x}^2\,dt\leq C. 
\end{equation*}
Noting also that
\begin{equation*}
\|\rho\|_{H^{k+1}_{x}}\leq \|\dxk\A\|_{L^{2}_x}+\|\dxk \A\|_{L^{2}_x}\|\rho\|_{H^{k}_x}+\|\rho\|_{H^{k}_x},
\end{equation*}
we obtain that 
\begin{equation}\label{eq:ubla}
\sup_{t\in (0,T^{*})}(\|\rho\|_{H^{k+1}_{x}}^2)+\int_0^{T^{*}}\|\dxkpo u\|_{L^{2}_x}^2\,dt\leq C. 
\end{equation}
Going back to the notation $(\re^{\lambda},\ue^{\lambda})$, \eqref{eq:ubla} implies that for any fixed $\e>0$
\begin{equation*}
\begin{aligned}
&\{\re^{\lambda}\}_{\lambda}\mbox{ is bounded in }L^{\infty}(0,T^{*};H^{k+1}(\T)),\\
&\{ \ue^{\lambda}\}_{\lambda}\mbox{ is bounded in }L^{2}(0,T^{*};H^{k+1}(\T)).
\end{aligned}
\end{equation*}
Since $k>3$ we can easily take the limit $\lambda\to 0$ and infer that there exists
\begin{equation}\label{eq:regstar}
(\re, \ue)\in C(0,T^{*};H^{k+1}(\T))\times L^{2}(0,T^{*};H^{k+1}(\T)), 
\end{equation}
solving \eqref{eq:app} in the classical sense. As a consequence, manipulating \eqref{eq:app} we also obtain that 
\begin{equation}\label{eq:appa2}
\begin{cases}
\partial_t\A_{\e}+\dx(\tilde{\kappa}(\re)\dx\ue)=-\dx(\ue\A_{\e})\\
-\deps\bar{\mu}(\re)\ddx\ue+\ue-\dx(\tilde{\kappa}(\re)\dx\A_{\e})=\frac{1}{2}\dx(\A_{\e}\A_{\e})
\end{cases}
\end{equation}
with $\A_{\e}$, $\tilde{\kappa}(\re)$, and $\bar{\mu}(\re)$ defined as in \eqref{eq:changevar}. Finally, by \eqref{eq:m4} and Sobolev embeddings, $(\re, \ue)$ satisfies 
\begin{equation}\label{eq:mfinal}
\begin{aligned}
&\partial_t(\|\dxk \A_{\e}\|_{L^{2}_x}^2+\|\re\|_{H^{k}_x}^2)+\|\dxk \ue\|_{L^{2}_x}^2+\deps\|\dxkpo \ue\|_{L^{2}_x}^2\\
\lesssim&(1+\|\A_{\e}\|_{L^{\infty}_x}^{2k-2}+\|\dx \A\|_{L^{\infty}_x}^2+\|\dx \ue\|_{L^{\infty}_x}^2+\|\A_{\e}\|_{L^{\infty}_x}^2\|\dx \ue\|_{L^{2}_x}^2)(\|\dxk\A_{\e}\|_{L^{2}_x}^2+\|\re\|_{H^{k}_x}^2).
\end{aligned}
\end{equation}
and by \eqref{eq:43-lambda}, it also holds that 
\begin{equation} \label{eq:43-nolambda}
    \begin{aligned}
    \int\kappa_\e(\re(t))\frac{|\del_x\re(t)|^2}{2}\,dx&+\int_0^t\int\re|u_{\e}|^2\,dxds+\deps\int_0^t\int|\del_x u_{\e}|^2\,dxds\\
    &=\int\kappa_\e(\re^0)\frac{|\del_x\re^0|^2}{2}\,dx.
\end{aligned}
\end{equation}

\noindent\emph{Step 7: Uniqueness.}\\

We prove that $(\re,\ue)$ is unique in this class. Contrary to {\em Step 4}, we need to use the formulation \eqref{eq:appa2}. Let $(\re, \ue)=(\rho, u)$ and $(\bar{\re},\bar{\ue})=(\bar{\rho}, \bar{\ue})$ be two solutions of \eqref{eq:app} with the same initial datum. By using \eqref{eq:appa2} we obtain 
\begin{equation*}
\begin{aligned}
\frac{d}{dt}\frac{\|\A-\bar{\A}\|_{L^{2}_x}^2}{2}&+\deps(\bar{\mu}(\rho)\dx(u-\bar{u}),\dx(u-\bar{u}))+\|u-\bar{u}\|_{L^{2}_x}^2\\
&=\frac{1}{2}(\dx u(\A-\bar{\A}), \A-\bar{\A})+(\dx(u-\bar{u})\bar{\A},\A-\bar{\A})\\
&+((\tilde{\kappa}(\rho)-\tilde{\kappa}(\bar{\rho}))\dx\bar{u},\A-\bar{\A})+((u-\bar{u})\dx\A,\A-\bar{\A})\\
&-\deps((\bar{\mu}(\rho)-\bar{\mu}(\bar{\rho}))\dx^2\bar{u},u-\bar{u})-((\tilde{\kappa}(\rho)-\tilde{\kappa}(\bar{\rho}))\dx\bar{\A},u-\bar{u})\\
&-\frac{1}{2}((\A-\bar{\A})(\A+\bar{\A}),\dx(u-\bar{u})).
\end{aligned}
\end{equation*}
Recalling \eqref{bounds} we obtain 
\begin{equation}\label{eq:uni1}
\begin{aligned}
&\frac{d}{dt}\frac{\|\A-\bar{\A}\|_{L^{2}_x}^2}{2}+\|u-\bar{u}\|_{H^1_x}
\lesssim a(t)(\|\A-\bar{\A}\|_{L^{2}_x}^2+\|\rho-\bar{\rho}\|_{L^{2}_x}^2),
\end{aligned}
\end{equation}
where 
\begin{equation*}
a(t):=(\|\dx u\|_{L^{\infty}_x}+\|\dx \bar{u}\|_{L^{\infty}_x}+\|\dx \bar{\A}\|^2_{L^{\infty}_{x}}+\|\dx \bar{u}\|^2_{L^{\infty}_{x}}+\|\A\|_{L^{\infty}_x}+\|\bar{A}\|_{L^{\infty}_x}). 
\end{equation*}
In addition, since 
\begin{equation}\label{eq:uni2}
\frac{d}{dt}\|\rho-\bar{\rho}\|_{L^{2}_x}^2\lesssim \|\rho-\bar{\rho}\|_{L^{2}_x}^2\|\dx u-\dx\bar{u}\|_{L^{2}_x}+\|\dx\bar{u}\|_{L^{\infty}_x}\|\rho-\bar{\rho}\|_{L^{2}_x}^2, 
\end{equation}
by adding \eqref{eq:uni1} and \eqref{eq:uni2}, and using Young's inequality and Sobolev embedding we obtain 
\begin{equation*}
\frac{d}{dt}\left(\|\A-\bar{\A}\|_{L^{2}_x}^2+\|\rho-\bar{\rho}\|_{L^{2}_x}^2\right)\lesssim (1+\|\rho\|_{H^2_x}^2+\|u\|_{H^{2}_x}^2)(\|\A-\bar{\A}\|_{L^{2}_x}^2+\|\rho-\bar{\rho}\|_{L^{2}_x}^2). 
\end{equation*}
Recalling that $(\rho, u)$ and $(\bar{\rho},\bar{u})$ have the same initial datum, by Gr\"onwall's lemma and \eqref{eq:regstar} we can conclude the uniqueness. \\

\noindent\emph{Step 8: Entropy estimate.}\\

Now that we sent $\lambda$ to zero, we can prove \eqref{eq:bd}. We multiply the first equation in \eqref{eq:app} by $\phie(\re)$ and the second equation by $\dx\phie(\re)$ and obtain
\begin{align}
&\int\partial_t\re\phie(\re)\,dx-\int\re\ue\dx\phie(\re)\,dx=0\label{eq:47}\\
&\int\re\ue\dx\phie(\re)\,dx+\int\re\me'(\re)|\ddx\phie(\re)|^2\,dx
=-\deps\int\ddx\ue\dx\phie(\re)\,dx.\label{eq:48} 
\end{align}
Summing up \eqref{eq:47} and \eqref{eq:48} we have 
\begin{equation}\label{eq:49}
\begin{aligned}
\int\Fe(\re)\,dx&+\int_0^t\int\re\me'(\re)|\ddx\phie(\re)|^2\,dx\,ds
=\int\Fe(\re^0)\,dx\\
&+\deps\int_0^t\int\dx\ue\ddx\phie(\re)\,dxds.
\end{aligned}
\end{equation}
We estimate the right-hand side of \eqref{eq:49}. By using \eqref{eq:lvaceps} and the choice of $\deps$, we have 
\begin{equation*}
\begin{aligned}
&\deps\int_0^t\int|\dx\ue||\ddx\phie(\re)|\,dxds
=\deps\int_0^t\int\frac{\dx\ue}{\sqrt{\re\me'(\re)}}\sqrt{\re\me'(\re)}\ddx\phie(\re)\,dxds\\
&\leq \frac{\sqrt{2}\deps^{\frac{1}{2}}e^{\frac{K}{4\e}}}{\e^2}
\left(\deps\int_0^t\int|\dx\ue|^2\,dxdt\right)^{\frac{1}{2}}\left(\e\int_0^t\int\re\me'(\re)|\ddx\phie(\re)|^2\,dxds\right)^{\frac{1}{2}}\\
&\leq C\,\e+\e\int_0^t\int\re\me'(\re)|\ddx\phie(\re)|^2\,dxds,
\end{aligned}
\end{equation*}
where in the last inequality we also used \eqref{eq:energy} and \eqref{eq:bd} follows.\\

\noindent\emph{Step 9: Extension of the solution to all times.}\\

Let $T>T^{*}$. The solution $(\re,\ue)$ can be extended up to time $T$, provided we can prove that there exists $C=C(T,\rho_0^{\e})>0$ such that 
\begin{equation}\label{eq:ext1}
\sup_{t\in(0,T)}\|\re(t)\|_{H^{k+1}_{x}}^2+\int_0^T\|\ue(t)\|_{H^{k+1}_{x}}^2\,dt\leq C. 
\end{equation}
Taking into account \eqref{eq:mfinal} and Sobolev embeddings, \eqref{eq:ext1} can be proved if we show that 
\begin{equation}\label{eq:ext2}
\begin{aligned}
&\A_{\e}\in L^{\infty}(0,T;H^{1}(\T)),\\
&\A_{\e}\in L^{2}(0,T;H^{2}(\T)),\\
&\ue\in L^{2}(0,T;H^{2}(\T)). 
\end{aligned}
\end{equation}
By using \eqref{eq:43-nolambda}, \eqref{bounds}, and \eqref{eq:49} we can infer that 
\begin{equation}\label{eq:ext3}
\begin{aligned}
&\A_{\e}\in L^{\infty}(0,T;L^2(\T))\cap L^{2}(0,T;H^{1}(\T)),\\
&\ue\in L^{2}(0,T;H^{1}(\T)),\\
&\rho,\rho^{-1}\in L^{\infty}(0,T;L^{\infty}(\T)). 
\end{aligned}
\end{equation}
To simplify the notation, we set $(\re,\ue)=(\rho,u)$. To prove \eqref{eq:ext2} we differentiate the first equation of \eqref{eq:appa2} and we multiply by $\dx\A$. Integrating on $\T$ we obtain 
\begin{equation}\label{eq:ext4}
\begin{aligned}
\frac{d}{dt}\frac{\|\dx\A\|_{L^{2}_x}^2}{2}-(\tilde{\kappa}(\rho)\dx^2 u,\dx^2 \A)&=-(u\dx^2\A,\dx\A)-2(\dx u\dx\A,\dx\A)\\
&-(\A\dx^2 u,\dx\A)-(\dx^2\tilde{\kappa}(\rho)\dx u,\dx\A)\\
&-(\dx\tilde{\kappa}(\rho)\dx^2 u,\dx\A). 
\end{aligned}
\end{equation}
Multiplying the second equation in \eqref{eq:appa2} by $-\dx^2 u$, we have 
\begin{equation}\label{eq:ext5}
\deps(\bar{\mu}(\rho)\dx^2 u,\dx^2 u)+\|\dx u\|_{L^{2}_x}^2+(\tilde{\kappa}(\rho)\dx^2\A,\dx^2 u)=-(\A\dx\A,\dx u)-(\dx\tilde{\kappa}(\rho)\dx\A,\dx^2\A). 
\end{equation}
Summing up \eqref{eq:ext4} and \eqref{eq:ext5} and using \eqref{eq:ext3}, Sobolev embeddings, and Young's inequality we obtain 
\begin{equation*}
\begin{aligned}
\frac{d}{dt}\|\dx\A\|_{L^{2}_x}^2+\|\dx^2 u\|_{L^{2}_x}^2&\lesssim\|\dx u\|_{L^{\infty}_x}\|\dx\A\|_{L^{2}_x}^2+\|\dx^2\tilde{\kappa}(\rho)\|_{L^{2}_x}^2\|\dx\A\|_{L^{2}_x}^2\\
&+\|\dx\tilde{\kappa}(\rho)\|^2_{L^{\infty}_x}\|\dx\A\|_{L^{2}_x}+\|\A\|_{L^{\infty}_x}^2\|\dx\A\|_{L^{2}_{x}}^2. 
\end{aligned}
\end{equation*}
Using again \eqref{eq:ext3} and Sobolev embeddings, we have that 
\begin{equation*}
\begin{aligned}
\|\dx^2\tilde{\kappa}(\rho)\|_{L^{2}_x}&\lesssim(1+\|\A\|_{L^{2}_x})\|\dx\A\|_{L^{2}_x}\lesssim\|\dx\A\|_{L^{2}_x},\\
\|\dx\tilde{\kappa}(\rho)\|_{L^{\infty}_x}&\lesssim \|\A\|_{L^{\infty}_x}\lesssim \|\dx\A\|_{L^{2}_x}.
\end{aligned}
\end{equation*}
Thus, by using Sobolev embeddings and Young's inequality we obtain 
\begin{equation*}
\frac{d}{dt}\|\dx\A\|_{L^{2}_x}^2+\|\dx^2 u\|_{L^{2}_x}^2\lesssim (1+\|\dx\A\|_{L^{2}_x}^2)\|\dx\A\|_{L^{2}_x}^2. 
\end{equation*}
Then, \eqref{eq:ext3} and Gr\"onwall's Lemma imply 
\begin{equation}\label{eq:freg}
\begin{aligned}
&\A\in L^{\infty}(0,T;H^{1}(\T)),\quad u\in L^{2}(0,T;H^{2}(\T)). 
\end{aligned}
\end{equation}
It remains to show that $\A\in L^{2}(0,T;H^{2}(\T))$. By using the second equation of \eqref{eq:appa2} we have that 
\begin{equation*}
-\tilde{\kappa}(\rho)\dx^2\A=\dx\tilde{\kappa}(\rho)\dx\A+\A\dx\A+
+\deps\bar{\mu}(\rho)\ddx u-u.
\end{equation*}
Thus, by using \eqref{eq:ext3} and Sobolev embedding, we have that 
\begin{equation*}
\|\dx^2\A\|_{L^{2}_x}^2\lesssim \|\dx\A\|_{2}^{4}+\|u\|_{L^{2}_x}^2+\|\dx^2 u\|_{L^{2}_{x}}^2. 
\end{equation*}
Then, by using \eqref{eq:freg} we have that $\A\in L^{2}(0,T;H^{2}(\T))$ and we can conclude. 
\end{proof}

\section{Proof of the main result}\label{sec:proof}
In this section we prove the main result of the paper. To minimize technicalities we make two additional assumptions. We assume first that $\re^0=\rho^0\in C^{\infty}(\T)$ with unit average and such that  $$\min_{x\in\T}\rho^0(x)=:m^0>0.$$
In particular, under this assumption, it holds that there exists a constant $C>0$ independent of $\e$ such that 
\begin{equation}\label{eq:51}
\int \kappa_\e(\rho^0)|\dx\rho^0|^2+\Fe(\rho^0)\,dx \leq C. 
\end{equation}
Moreover, to avoid considering the logarithm in the definition of $\phie$, $\me$, $\Fe$, and $\kappa_\e$, in addition to $\beta>-3$ we also assume $\beta\not\in\{-2, -5/3, -3/2, -1\}$. Finally, in this section the symbol $\lesssim$ means that the terms of the implicit constant in the inequalities is independent of $\e$.
\subsection{Uniform bounds}
We prove some bounds uniform in $\e$ implied by \eqref{eq:energy} and \eqref{eq:bd}. 
\begin{lemma}\label{lem:51}
There exists a constant $C>0$, independent of $\e$, such that 
\begin{align}
&\|\dx\re^{\frac{\beta+2}{2}}\|_{L^{\infty}_t(L^2_x)}\leq C,\label{eq:52}\\
&\|\re\|_{L^{\infty}_t(L^\infty_x)}\leq C\label{eq:53}.
\end{align}
Moreover, if $-3<\beta<-2$ then
\begin{equation}\label{eq:53p}
\|\re^{-1}\|_{L^{\infty}_t(L^\infty_x)}\leq C.
\end{equation}
\end{lemma}
\begin{proof}
By the definition of $\kappa_\e$, \eqref{eq:energy}, and \eqref{eq:51} we have that for some $C>0$ independent of $\e$ it holds that 
\begin{equation*}
\sup_{t}\int\re^{\beta}|\dx\re|^2\,dx\leq C,
\end{equation*}
and thus \eqref{eq:52} follows easily. Concerning \eqref{eq:53}, we first note that since the average is conserved and $\rho^0$ has unit average, we have that for any $(t,x)\in (0,T)\times\T$ 
\begin{equation*}
\re^{\frac{\beta+2}{2}}(t,x)\leq 1+\left(\int|\dx\re^{\frac{\beta+2}{2}}|^2\,dx\right)^{\frac{1}{2}}. 
\end{equation*}
Then, if $\beta>-2$ we obtain \eqref{eq:53} while if $-3<\beta<-2$ we obtain \eqref{eq:53p}. It remains only to prove \eqref{eq:53} if $-3<\beta<-2$. We have that 
\begin{equation*}
\int|\dx\re^{\frac{\beta+3}{2}}|\,dx\les\int\re^{\frac{1}{2}}|\dx\re^{\frac{\beta+2}{2}}|\,dx\les \left(\int\re\,dx\right)^{\frac{1}{2}}\left(\int|\dx\re^{\frac{\beta+2}{2}}|^2\,dx\right)^{\frac{1}{2}}.
\end{equation*}
Thus, 
\begin{equation*}
\{\dx\re^{\frac{\beta+3}{2}}\}_{\e}\mbox{ is bounded in }L^{\infty}(0,T;L^{1}(\T)), 
\end{equation*}
and we obtain \eqref{eq:53} arguing as above.
\end{proof}
In the next lemma we prove some bounds on the second derivative of certain powers of the density. 
\begin{lemma}\label{lem:52}
There exists a constant $C>0$ such that for $\theta=\frac{3\beta+5}{4}$ it holds that 
\begin{equation}\label{eq:54}
\begin{aligned}
&\|\ddx\re^\theta\|_{L^{2}_{t,x}}\leq C,&\quad&\|\dx\re^{\frac{\theta}{2}}\|_{L^{4}_{t,x}}\leq C.&
\end{aligned}
\end{equation}
Moreover, 
\begin{align}
&\e\iint|\ddx\re^{\frac{2\beta+3}{4}}|^2\,dxdt\leq C,&
&\e\iint|\dx\re^{\frac{2\beta+3}{8}}|^4\,dxdt\leq C,&\label{eq:55}\\
&\e^3\iint|\ddx\re^{-\frac{1}{4}}|^2\,dxdt\leq C,&
&\e^3\iint|\dx\re^{-\frac{1}{8}}|^4\,dxdt\leq C.&\label{eq:56}
\end{align}
\end{lemma}
\begin{proof}
By the definition of $\phie(\re)$ we have that 
\begin{equation}\label{eq:57}
\phie'(\re)=\re^{\frac{\beta-1}{2}}+\e\re^{-\frac{3}{2}}.
\end{equation}
By using \eqref{eq:bd} and Lemma \ref{lem:estkor} we can infer that there exists a constant $C>0$ such that 
\begin{equation*}
\iint\re^{2}(\phie'(\re))^3\left(|\ddx\re|^2+\frac{|\dx\re|^4}{\re^2}\right)\,dxdt\leq C. 
\end{equation*}
Since $\phie'(\re)\geq \re^{\frac{\beta-1}{2}}$ we have that 
\begin{equation*}
\iint\re^{\frac{3\beta+1}{2}}\left(|\ddx\re|^2+\frac{|\dx\re|^4}{\re^2}\right)\,dxdt\leq C,
\end{equation*}
and thus by Remark \ref{rem:sd} we deduce \eqref{eq:54}. The bounds \eqref{eq:55} and \eqref{eq:56} are obtained in a similar way. 
\end{proof}
\subsection{Convergence lemma}
We start by showing the a.e. convergence of $\{\re\}_{\e}$
\begin{lemma}\label{lem:53}
There exists $\re\geq 0$ such that, up to a subsequence, 
\begin{equation}\label{eq:58}
\re\rightarrow \rho\mbox{ a.e. on }(0,T)\times\T. 
\end{equation}
Moreover, 
\begin{equation}\label{eq:58p}
\re\rightarrow \rho\mbox{ strongly in }L^p(0,T;L^{p}(\T))\mbox{ for any }p<\infty,
\end{equation}
and if $-3<\beta<-2$ 
\begin{equation}\label{eq:58dp}
\frac{1}{\re}\rightarrow\frac{1}{\rho}\mbox{ strongly in }L^p(0,T;L^{p}(\T))\mbox{ for any }p<\infty.
\end{equation}
\end{lemma}
\begin{proof}
We recall that $\theta=\frac{3\beta+5}{4}$. In particular, if $\beta>-3$ then $\theta+1>0$. Moreover, by Poincar\`e inequality we have that 
\begin{equation*}
\{\dx\re^\theta\}_{\e}\mbox{ is bounded in }L^2(0,T;L^2(\T)),
\end{equation*}
and then we easily get that 
\begin{equation}\label{eq:59}
\{\re^{\theta+1}\}_{\e}\mbox{ is bounded in }L^2(0,T;H^1(\T)).
\end{equation}
Moreover, for any $\psi\in C^{\infty}(\T)$, we have that 
\begin{equation*}
|\ddx(\re^\theta\psi)|\les |\ddx\re^\theta||\psi|+|\dx\\re^{\theta}||\dx\psi|+\re^{\theta}|\ddx\psi|.
\end{equation*}
In particular, by using \eqref{eq:53} when $\theta>0$ and \eqref{eq:53p} when $\theta<0$, we have that 
\begin{equation*}
\|\ddx(\re^{\theta}\psi\|_{L^{2}_{t,x}}\les\|\re^{\theta}\|_{L^2_t(H^2_x)}\|\psi\|_{W^{2,\infty}_{x}}.
\end{equation*}
Finally, by using \eqref{eq:energy} and \eqref{eq:bd} we get 
\begin{equation}\label{eq:510}
\begin{aligned}
\{\sqrt{\re\me'(\re)}\ddx\phie(\re)\}_{\e}&\mbox{ is bounded in }L^2(0,T;L^2(\T))\\
\{\sqrt{\deps}\dx\ue\}_{\e}&\mbox{ is bounded in }L^2(0,T;L^2(\T)).
\end{aligned}
\end{equation}
Then, 
\begin{equation*}
\begin{aligned}
\int|\langle\partial_t\re^{\theta+1},\psi\rangle\,dt
&\les\iint\sqrt{\re\me'(\re)}\sqrt{\re\me'(\re)}|\ddx\phie(\re)||\ddx(\re^{\theta}\psi)|\,dxdt\\
&\iint\deps|\dx\ue|\ddx(\re^{\theta}\psi)|\,dxdt\\
&\les\|\sqrt{\re\me'(\re)}\|_{L^{\infty}_{t,x}}\|\sqrt{\re\me'(\re)}\ddx\phie(\re)\|_{L^{2}_{t,x}}\|\re^{\theta}\|_{L^2_t(H^2_x)}\|\psi\|_{W^{2,\infty}_x}\\
&+\sqrt{\deps}\|\sqrt{\deps}\dx\ue\|_{L^{2}_{t,x}}|\re^{\theta}\|_{L^2_t(H^2_x)}\|\psi\|_{W^{2,\infty}_x}\\
&\les 1+\sqrt{\deps},
\end{aligned}
\end{equation*}
and by using \eqref{eq:59}, \eqref{eq:510}, the definition of $\me$, \eqref{eq:58}, and Sobolev embeddings we can conclude that 
\begin{equation*}
\{\partial_t\re^{\theta+1}\}_{\e}\mbox{ is bounded in }L^{1}(0,T;H^{-3}(\T)).
\end{equation*}
Therefore, by Aubin-Lions Lemma we deduce that 
\begin{equation*}
\{\re^{\theta+1}\}_{\e}\mbox{ is pre-compact in }L^{2}(0,T;L^2(\T)).
\end{equation*}
Thus, recalling that $\theta+1>0$, we can conclude the proof of \eqref{eq:58}. Finally, \eqref{eq:58p} follows by \eqref{eq:58} and \eqref{eq:53} and \eqref{eq:58dp} follows by \eqref{eq:58} and \eqref{eq:53p}. 
\end{proof}
Next, we study the convergence of some quantities involving derivatives of powers of the density. 
\begin{lemma}\label{lem:54}
It holds that 
\begin{align}
&\ddx\re^\theta\weakto\ddx\rho^\theta\mbox{ weakly in }L^2(0,T;L^2(\T)),\label{eq:511}\\
&\re^{\delta}\dx\re^{\frac{\theta}{2}}\rightarrow \rho^{\delta}\dx\rho^{\frac{\theta}{2}}\mbox{ strongly in }L^p(0,T;L^p(\T))
\mbox{ for any }p\in[1,4)\mbox{ and }\delta>0. \label{eq:512}
\end{align}
\end{lemma}
\begin{proof}
It follows from \eqref{eq:54} that the sequence $\{\ddx\re\}_{\e}$ is weakly compact in $L^{2}(0,T;L^{2}(\T))$. Then, by using \eqref{eq:58}, the convergence \eqref{eq:511} follows. Concerning \eqref{eq:511}, by the very same argument we can infer that 
\begin{equation}\label{eq:513}
\dx\re^{\frac{\theta}{2}}\weakto\dx\rho^{\frac{\theta}{2}}
\mbox{ weakly in }L^{p}(0,T;L^{p}(\T))\mbox{ for any }1\leq p\leq 4,
\end{equation}
and therefore $\dx\rho^{\theta}\in L^{4}(0,T;L^{4}(\T))$. It follows from \eqref{eq:513}, \eqref{eq:58}, and \eqref{eq:53}, that 
\begin{equation}\label{eq:514}
\re^{\delta}\dx\re^{\frac{\theta}{2}}\weakto\rho^{\delta}\dx\rho^{\frac{\theta}{2}}\mbox{ in }L^{2}(0,T;L^{2}(\T))\mbox{ for any }\delta>0. 
\end{equation}
To update \eqref{eq:514} from weak convergence to strong convergence we show the convergence of the norms. We note that 
\begin{equation*}
\iint\re^{2\delta}|\dx\re^{\frac{\theta}{2}}|^2\,dxdt
=\frac{\theta}{8\delta}\iint\dx\re^{2\delta}\dx\re^{\theta}\,dxdt=-\frac{\theta}{8\delta}\iint\re^{2\delta}\ddx\re^{\theta}\,dxdt. 
\end{equation*}
By using \eqref{eq:58} and \eqref{eq:511}, since $\delta>0$, we have that 
\begin{equation*}
-\frac{\theta}{8\delta}\iint\re^{2\delta}\ddx\re^{\theta}\,dxdt\rightarrow -\frac{\theta}{8\delta}\iint\rho^{2\delta}\ddx\rho^{\theta}\,dxdt=\iint\rho^{2\delta}|\dx\rho^{\frac{\theta}{2}}|^2\,dxdt. 
\end{equation*}
Thus, we have proven that 
\begin{equation*}
\iint \re^{2\delta}|\dx\re^{\frac{\theta}{2}}|^2\,dxdt\rightarrow \iint\rho^{2\delta}|\dx\rho^{\frac{\theta}{2}}|^2\,dxdt,
\end{equation*}
which, together with \eqref{eq:514}, implies that 
\begin{equation*}
\re^{\delta}\dx\re^{\frac{\theta}{2}}\rightarrow \rho^{\delta}\dx\rho^{\frac{\theta}{2}}\mbox{ strongly in }L^{2}(0,T;L^{2}(\T)). 
\end{equation*}
Finally, using \eqref{eq:53} and \eqref{eq:54} we have that 
\begin{equation*}
\{\re^{\delta}\dx\re^{\frac{\theta}{2}}\}_{\e}\mbox{ is bounded in }L^{4}(0,T;L^{4}(\T)),
\end{equation*}
and therefore 
\begin{equation*}
\re^{\delta}\dx\re^{\frac{\theta}{2}}\rightarrow \rho^{\delta}\dx\rho^{\frac{\theta}{2}}\mbox{ in }L^{p}(0,T;L^{p}(\T))\mbox{ for any }1\leq p< 4. 
\end{equation*}
\end{proof}
\subsection{Proof of the main theorem}
We are now in a position to prove the main result of this paper. 
\begin{proof}[Proof of Theorem \ref{teo:main}]
Let $\psi\in C^{\infty}_{c}([0,T)\times\T)$. Multiplying the first equation of \eqref{eq:app} by $\psi$ and the second equation by $\dx\psi$, integrating by parts and summing up we obtain 
\begin{equation}\label{eq:515}
\begin{aligned}
-\iint\re\partial_t\psi\,dxdt&+\iint\re\me'(\re)\ddx\phie(\re)\ddx\psi\,dxdt\\
&\deps\iint\dx\ue\ddx\psi\,dxdt-\int\rho^0\psi(0)\,dx=0
\end{aligned}
\end{equation}
By using \eqref{eq:58p} we have that as $\e\to 0$
\begin{equation*}
\iint\re\partial_t\psi\,dxdt\rightarrow \iint\rho\partial_t\psi\,dxdt. 
\end{equation*}
By using Cauchy-Schwartz inequality and \eqref{eq:energy} we have 
\begin{equation*}
\deps\iint|\dx\ue||\ddx\psi|\,dxdt\les \sqrt{\deps}\rightarrow 0 
\end{equation*}
as $\e\to 0$. It remains to consider the second term in \eqref{eq:515}. By using the definition of $\phie$ and $\me$ we have that 
\begin{equation*}
\begin{aligned}
\iint\re\me'(\re)\ddx\phie(\re)\ddx\psi\,dxdt&=\frac{2}{\beta+1}\iint\re^{\frac{\beta+3}{2}}\ddx\re^{\frac{\beta+1}{2}}\ddx\psi\,dxdt-\e\iint\re^{\frac{\beta+3}{2}}\ddx\re^{-\frac{1}{2}}\ddx\psi\,dxdt\\
&+\frac{\e}{\beta+1}\iint\re^{\frac{1}{2}}\ddx\re^{\frac{\beta+1}{2}}\ddx\psi\,dxdt-\frac{e^2}{2}\iint\re^{\frac{1}{2}}\ddx\re^{-\frac{1}{2}}\ddx\psi\,dxdt\\
&=I^{\e}_{1}+I^{\e}_{2}+I^{\e}_{3}+I^{\e}_{4}.
\end{aligned}
\end{equation*}
Let us consider $I^{\e}_1$. By a direct calculation we have that
\begin{equation*}
\begin{aligned}
\frac{2}{\beta+1}\iint\re^{\frac{\beta+3}{2}}\ddx\re^{\frac{\beta+1}{2}}\ddx\psi\,dxdt
&=\frac{1}{\theta}\iint\re^{\beta+2-\theta}\ddx\re^{\theta}\ddx\psi\,dxdt\\
&-\frac{\beta+3}{\theta^2}\iint\re^{\beta+2-\theta}|\dx\re^{\frac{\theta}{2}}|^2\ddx\psi\,dxdt.
\end{aligned}
\end{equation*}
Note that if $\beta>-3$, the $\beta+2-\theta>0$. Then, by using \eqref{eq:58p} and \eqref{eq:511} we have that 
\begin{equation*}
\iint\re^{\beta+2-\theta}\ddx\re^{\theta}\ddx\psi\,dxdt
\rightarrow \iint\rho^{\beta+2-\theta}\ddx\rho^{\theta}\ddx\psi\,dxdt
\end{equation*}
as $\e\to 0$. Moreover, by using \eqref{eq:412} we have that 
\begin{equation*}
\iint\re^{\beta+2-\theta}|\dx\re^{\frac{\theta}{2}}|^2\ddx\psi\,dxdt
\rightarrow \iint\rho^{\beta+2-\theta}|\dx\rho^{\frac{\theta}{2}}|^2\ddx\psi\,dxdt
\end{equation*}
as $\e\to 0$. Thus, as $\e\to 0$
\begin{equation*}
I^1_{\e}\rightarrow \frac{1}{\theta}\iint\rho^{\beta+2-\theta}\ddx\rho^{\theta}\ddx\psi\,dxdt-\frac{\beta+3}{\theta^2}\iint\rho^{\beta+2-\theta}|\dx\rho^{\frac{\theta}{2}}|^2\ddx\psi\,dxdt.
\end{equation*}
Next, we consider $I^2_{\e}$. By a direct calculation we have 
\begin{equation*}
\begin{aligned}
|I^2_{\e}|&\les\|\ddx\psi\|_{L^{\infty}_x}\e\iint\re^{\frac{1}{4}}(|\ddx\re^{\frac{2\beta+3}{4}}|
+|\dx\re^{\frac{2\beta+3}{8}}|^2)\,dxdt\\
&\les\|\ddx\psi\|_{L^{\infty}_x}\left(\iint \re^{\frac{1}{2}}\,dxdt\right)^{\frac{1}{2}}
\left(\e^2\iint|\ddx\re^{\frac{2\beta+3}{4}}|^2
+|\dx\re^{\frac{2\beta+3}{8}}|^4\,dxdt\right)^{\frac{1}{2}}\\
&\les\|\ddx\psi\|_{L^{\infty}_x}\sqrt{\e}\left(\iint \re^{\frac{1}{2}}\,dxdt\right)^{\frac{1}{2}}
\left(\e\iint|\ddx\re^{\frac{2\beta+3}{4}}|^2
+|\dx\re^{\frac{2\beta+3}{8}}|^4\,dxdt\right)^{\frac{1}{2}}\\
&\les\|\ddx\psi\|_{L^{\infty}_x} \sqrt{\e}, 
\end{aligned}
\end{equation*}
where in the last line we used \eqref{eq:53} and \eqref{eq:55}. The term $I^{3}_{\e}$ is also treated similarly. Indeed, by using \eqref{eq:53} and \eqref{eq:55}, we have that 
\begin{equation*}
\begin{aligned}
|I^{3}_{\e}|&\les\|\ddx\psi\|_{L^{\infty}_x}\sqrt{\e}\left(\iint \re^{\frac{1}{2}}\,dxdt\right)^{\frac{1}{2}}
\left(\e\iint|\ddx\re^{\frac{2\beta+3}{4}}|^2
+|\dx\re^{\frac{2\beta+3}{8}}|^4\,dxdt\right)^{\frac{1}{2}}\\
&\les\|\ddx\psi\|_{L^{\infty}_x}\sqrt{\e}. 
\end{aligned}
\end{equation*}
Finally, we consider $I^4_{\e}$. We have that 
\begin{equation*}
\begin{aligned}
|I^4_{\e}|&\les|\ddx\psi\|_{L^{\infty}_x}\e^3\iint\re^{\frac{1}{4}}(|\ddx\re^{-\frac{1}{4}}|
+|\dx\re^{-\frac{1}{8}}|^2)\,dxdt\\
&\les\|\ddx\psi\|_{L^{\infty}_x}\sqrt{\e}\left(\iint \re^{\frac{1}{2}}\,dxdt\right)^{\frac{1}{2}}
\left(\e^3\iint|\ddx\re^{-\frac{1}{4}}|^2
+|\dx\re^{-\frac{1}{8}}|^4\,dxdt\right)^{\frac{1}{2}}\\
&\les\|\ddx\psi\|_{L^{\infty}_x}\sqrt{\e}
\end{aligned}
\end{equation*}
where in the last line we used \eqref{eq:53} and \eqref{eq:56}. Thus, 
in conclusion
\begin{equation*}
\begin{aligned}
\iint\re\me'(\re)\ddx\phie\ddx\psi\,dxdt\rightarrow&\frac{1}{\theta}\iint\rho^{\beta+2-\theta}\ddx\rho^{\theta}\ddx\psi\,dxdt\\
&-\frac{\beta+3}{\theta^2}\iint\rho^{\beta+2-\theta}|\dx\rho^{\frac{\theta}{2}}|^2\ddx\psi\,dxdt,
\end{aligned}
\end{equation*}
as $\e\to 0$ and the proof is complete. 
\end{proof}
\appendix
\section{\empty}
For the sake of completeness, we give the full proof of inequality \eqref{eq:bdentropyintro2}. 
\begin{lemma}\label{lem:inequality}
    Let $\beta\in\R\setminus\{-\frac{5}{3}, -1\}$. There exists a constant $C>0$, depending on $\beta$ such that
    \begin{equation}\label{ineq}
        \int |\del_x^2 \rho^\theta|^2 \,dx \leq C \int \rho^{\frac{\beta+3}{2}}|\del_x^2\rho^{\frac{\beta+1}{2}}|^2 \,dx,
    \end{equation}
    \noindent where $\theta = \frac{3\beta+5}{4}$.
\end{lemma}

\begin{proof}
Let $\beta\in\R\setminus\{-\frac{5}{3}, -1\}$ and $\theta:=(3\beta+5)/4$. We first note that 
\begin{equation*}
\ddx\rho^{\frac{\beta+1}{2}}=\left(\frac{\beta+1}{2\theta}\right)
\rho^{\frac{\beta+1}{2}-\theta}\left(\ddx\rho^{\theta}-\frac{\beta+3}{\theta}|\dx\rho^{\frac{\theta}{2}}|^2\right).
\end{equation*}
Thus
\begin{equation*}
\int\rho^{\frac{\beta+3}{2}}|\ddx\rho^{\frac{\beta+1}{2}}|^2\,dx
=\left(\frac{\beta+1}{2\theta}\right)^2
\int\left(\ddx\rho^{\theta}-\frac{\beta+3}{\theta}|\dx\rho^{\frac{\theta}{2}}|^2\right)^2\,dx.
\end{equation*}
Expanding the square we have 
\begin{equation*}
\begin{aligned}
\int|\ddx\rho^{\theta}|^2\,dx+\frac{(\beta+3)^2}{\theta^2}\int|\dx\rho^{\frac{\theta}{2}}|^4\,dx&=\left(\frac{2\theta}{\beta+1}\right)^2\int\rho^{\frac{\beta+3}{2}}|\ddx\rho^{\frac{\beta+1}{2}}|^2\,dx\\
&+\frac{2(\beta+3)}{\theta}\int\ddx\rho^{\theta}|\dx\rho^{\frac{\theta}{2}}|^2\,dx. 
\end{aligned}
\end{equation*}
Integrating by parts we have that 
\begin{equation*}
\int\ddx\rho^{\theta}|\dx\rho^{\frac{\theta}{2}}|^2\,dx
=\frac{4}{3}\int|\dx\rho^{\frac{\theta}{2}}|^4\,dx,
\end{equation*}
and thus
\begin{equation*}
\int|\ddx\rho^{\theta}|^2\,dx+\left(\frac{(\beta+3)^2}{\theta^2}-\frac{8(\beta+3)}{3\theta}\right)\int|\dx\rho^{\frac{\theta}{2}}|^4\,dx=\left(\frac{2\theta}{\beta+1}\right)^2\int\rho^{\frac{\beta+3}{2}}|\ddx\rho^{\frac{\beta+1}{2}}|^2\,dx.
\end{equation*}
Finally, recalling the inequality 
\begin{equation*}
\frac{16}{9}\int|\dx\rho^{\frac{\theta}{2}}|^4\,dx\leq \int|\ddx\rho^{\theta}|^2\,dx,
\end{equation*}
see \cite{Be96}, we deduce that for any given $\beta\not\in\{-5/3\}$ 
\begin{equation*}
\frac{16}{9}+\frac{(\beta+3)^2}{\theta^2}-\frac{8(\beta+3)}{3\theta}>0,
\end{equation*}
and so there exists a constant $C=C(\beta)>0$ such that \eqref{ineq} holds. 
    
\end{proof}

\end{document}